\documentclass{article}
\title{Stability of geodesics in the Brownian map}
\author{Omer Angel \and Brett Kolesnik \and Gr{\'e}gory Miermont}
\date{}

\AtEndDocument{
  \bigskip
  \small
  \textsc{Omer Angel, Brett Kolesnik} \par
  \noindent\textsc{Department of Mathematics, University of British Columbia} \par
  \textit{Email:} \texttt{\{angel,bkolesnik\}@math.ubc.ca}\par 
  \medskip
  \textsc{Gr\'egory Miermont}\par 
  \noindent\textsc{Unit{\'e} de Math{\'e}matiques Pures et Appliqu{\'e}es, \'{E}cole Normale Sup\'{e}rieure de Lyon 
\& Institut Universitaire de France} \par
  \textit{Email:} \texttt{gregory.miermont@ens-lyon.fra}
}

\usepackage[T1]{fontenc}
\usepackage{lmodern,microtype,hyperref,graphicx}
\usepackage{amsthm,amsmath,amssymb}
\usepackage{enumitem}
\usepackage[font=footnotesize,margin=2.5cm]{caption}

\theoremstyle{plain}
\newtheorem{lem}{Lemma}
\newtheorem{pro}[lem]{Proposition}
\newtheorem{thm}[lem]{Theorem}
\newtheorem{cor}[lem]{Corollary}
\theoremstyle{definition}
\newtheorem*{defn}{Definition}
\newtheorem*{con}{Conjecture}

\newcommand{\g}{\gamma}
\newcommand{\lam}{\lambda}
\newcommand{\eps}{\epsilon}

\renewcommand{\d}{d}
\newcommand{\m}{M}
\newcommand{\ee}{{\bf e}}
\newcommand{\Te}{{\cal T}_\ee}
\newcommand{\TZ}{{\cal T}_Z}
\newcommand{\dd}{d_\ee}
\newcommand{\pp}{p_\ee}
\newcommand{\PP}{{\bf p}}

\newcommand{\geo}{G}
\newcommand{\mul}{S}
\newcommand{\cut}{C}
\newcommand{\net}{N}
\renewcommand{\star}{Z}
\newcommand{\pair}{P}

\newcommand{\NN}{{\mathbb N}}
\newcommand{\RR}{{\mathbb R}}
\newcommand{\CC}{{\mathbb C}}
\newcommand{\HH}{{\mathbb H}}
\renewcommand{\SS}{{\mathbb S}}

\renewcommand{\H}{\dim}
\newcommand{\M}{\mathrm{Dim}\,}
\newcommand{\pack}{\dim_{\rm P}}
\newcommand{\uM}{\overline{\mathrm{Dim}}\,}
\newcommand{\lM}{\underline{\mathrm{Dim}}\,}

\newcommand{\tto}{\rightrightarrows}

\begin{document}
%
%
%
%
%
%
%
%
%

\maketitle

\begin{abstract}
The Brownian map is a random geodesic metric space 
arising as the scaling limit of random planar maps. We 
strengthen the so-called \emph{confluence of geodesics} 
phenomenon observed at the root of the map, and with this, 
reveal several properties of its rich geodesic structure.

Our main result is the continuity of the cut locus at typical 
points. A small shift from such a point results in a small, 
local modification to the cut locus. Moreover, the cut locus 
is uniformly stable, in the sense that any two cut loci coincide 
outside a closed, nowhere dense set of zero measure.

We obtain similar stability results for the set of points inside 
geodesics to a fixed point. Furthermore, we show that the 
set of points inside geodesics of the map is of first Baire 
category. Hence, most points in the Brownian map are 
endpoints.

Finally, we classify the types of geodesic networks which 
are dense. For each $k\in\{1,2,3,4,6,9\}$, there is a dense 
set of pairs of points which are joined by networks of exactly 
$k$ geodesics and of a specific topological form. We find the 
Hausdorff dimension of the set of pairs joined by each type 
of network. All other geodesic networks are nowhere dense.
\end{abstract}

%
%

\section{Introduction}\label{S_intro}

A universal scaling limit of random planar maps has recently 
been identified by Le Gall~\cite{LG13} (triangulations and 
$2k$-angulations, $k>1$) and Miermont~\cite{M13} 
(quadrangulations) as a random geodesic metric space called 
the \emph{Brownian map} $(\m,\d)$. In this work, we establish 
properties of the Brownian map which are a step towards a 
complete understanding of its geodesic structure.

The works of Cori and Vauquelin~\cite{CV81} and 
Schaeffer~\cite{S98} describe a bijection from well-labelled plane 
trees to rooted planar maps. The Brownian map is obtained as 
a quotient of Aldous'~\cite{A91,A93} \emph{continuum random tree}, 
or CRT, by assigning Brownian labels to the CRT and then 
identifying some of its non-cut-points, or \emph{leaves}, according 
to a continuum analogue of the \emph{CVS-bijection} 
(see Section~\ref{S_BM}). The resulting object is homeomorphic 
to the sphere $\SS^2$ (Le Gall and Paulin~\cite{LGP08} and 
Miermont~\cite{M08}) and of Hausdorff dimension 4 
(Le Gall~\cite{LG07}) and is thus in a sense a random, 
fractal, spherical surface. 

Le Gall~\cite{LG10} classifies the geodesics to the root, which is a 
certain distinguished point of the Brownian map 
(see Section~\ref{S_BM}), in terms of the label
process on the CRT (see Section~\ref{S_simple}). Moreover, the
Brownian map is shown to be invariant in distribution under uniform
re-rooting from the volume measure $\lam$ on $\m$ (see Section
\ref{S_BM}). Hence, geodesics to typical points exhibit a similar
structure as those to the root. It thus remains to investigate
geodesics from special points of the Brownian map.

\subsection{Geodesic nets}\label{S_net}

A striking consequence of Le Gall's description of geodesics to the 
root is that any two such geodesics are bound to meet and then 
coalesce before reaching the root, a phenomenon referred to as the
\emph{confluence of geodesics} (see Section~\ref{S_at}).
In fact, the set of points in the relative interior of a geodesic to the 
root is a small subset which is homeomorphic to an $\RR$-tree and 
of Hausdorff dimension 1 (see~\cite{LG10}).

\begin{defn}
  We call a subset $\g\subset\m$ a \emph{geodesic segment} if
  $(\g,\d)$ is isometric to a compact interval. The \emph{extremities} of the
  geodesic segment are the images, say $x$ and $y$, of the extremities
  of the source interval, and we say that $\g$ is a geodesic segment
  between $x$ and $y$ (or from $x$ to $y$ if we insist on
  distinguishing one orientation of $\g$). 
\end{defn}

We will often denote a particular geodesic segment between $x,y\in\m$
as $[x,y]$, and denote its relative interior by $(x,y)=[x,y]-\{x,y\}$.
(Since there might be more than one such geodesic segment, we will
be careful in lifting any ambiguity that might arise from this
notation.) We define $[x,y)$ and $(x,y]$ similarly. 

\begin{defn}
  For $x\in\m$, the \emph{geodesic net} of $x$, denoted $\geo(x)$, 
  is the set of points $y\in\m$ that are contained in the relative 
  interior of a geodesic segment to $x$.
\end{defn}

Although geodesics to the root of the Brownian map 
are understood,
the structure of geodesics to general points 
remains largely mysterious. 
Indeed, the main obstacle in establishing the
existence of the Brownian map
is to relate a geodesic
between a pair of typical points
to geodesics to the root.
A compactness argument of Le Gall~\cite{LG07} 
yields scaling limits of planar maps 
along subsequences, however the question
of uniqueness remained unresolved for some time.
Finally, making use of 
Le Gall's description of geodesics
to the root, Le Gall~\cite{LG13} and 
Miermont~\cite{M13} show that distances to the root
provide enough information to characterize
the Brownian map metric.
Let $\g$ be 
a geodesic between points selected uniformly according 
to $\lambda$. (By the confluence of geodesics phenomenon,
the root of the map is almost surely disjoint from $\g$.)
In~\cite{LG13,M13} the set of points
$z\in\g$ such that the relative interior of any
geodesic from $z$ to the root is disjoint from $\g$
is shown to be small compared to $\g$.
Hence, roughly speaking, 
``most'' points in ``most'' geodesics of the Brownian map 
are in a geodesic to the root.
(See the discussion around equation (2) in \cite{LG13} 
and \cite[Section 2.3]{M13} for precise statements.)

In this work, we show that for {\it any} two points $x,y\in\m$, 
points which are in a geodesic to $x$ but not in a geodesic to 
$y$ are exceptional. Hence, to a considerable extent, 
the geodesic 
structure of the Brownian map is similar as viewed from any 
point of the map, providing further evidence that it is, 
to quote Le Gall~\cite{LG14}, 
  ``very regular in its irregularity.''

\begin{thm}\label{T_net-nwd}
 Almost surely, for all $x,y\in\m$, 
 $\geo(x)$ and $\geo(y)$ coincide outside a 
 closed, nowhere dense set
 of zero $\lambda$-measure. 
\end{thm}

Furthermore, for most points $x\in\m$, the effect of small 
perturbations of $x$ on $\geo(x)$ is localized.

\begin{thm}\label{T_net-cts}
  Almost surely, the function $x\mapsto\geo(x)$ is continuous 
  almost everywhere in the following sense.

  For $\lam $-almost every $x\in\m$, for any neighbourhood 
  $N$ of $x$, there is a sub-neighbourhood $N'\subset N$ 
  so that 
  $\geo(x')-N$ is the same for all $x'\in N'$.
  \end{thm}

The \emph{uniform infinite planar triangulation}, or UIPT, 
introduced by Angel and Schramm~\cite{AS03}, is a random lattice 
which arises as the \emph{local limit} of random triangulations 
of the sphere. The case of quadrangulations, giving rise to 
the UIPQ, is due to Krikun~\cite{K05}.
We remark that Theorem~\ref{T_net-cts} is in a sense a 
continuum analogue to a result of Krikun~\cite{K08} 
(see also Curien, M{\'e}nard, and Miermont~\cite{CMM13})  
which shows that the ``Schaeffer's tree'' of the UIPQ only 
changes locally after relocating its root. 

Next, we find that the union of all geodesic nets 
is relatively small.

\begin{defn}
  Let $F=\bigcup_{x\in\m}\geo(x)$ denote the set of 
  points in the relative interior of a geodesic in $(\m,\d)$. 
  We refer to $F$ as the \emph{geodesic framework} and 
  $E=F^c$ as the \emph{endpoints of the Brownian map}.
\end{defn}

\begin{thm}\label{T_frame}
  Almost surely, the geodesic framework of the Brownian 
  map, $F\subset\m$, is of first Baire category. 
\end{thm}

Hence, the endpoints of the Brownian map, $E\subset\m$, 
is a residual subset. This property of the 
Brownian map is reminiscent of a result of 
Zamfirescu~\cite{Z82}, which states that for 
most convex surfaces --- that is, for all surfaces in a 
residual subset of the Baire space of convex surfaces 
in $\RR^n$ endowed with the Hausdorff metric --- the 
endpoints form a residual set.

\subsection{Cut loci}\label{S_cut}

Recall that the cut locus of a point $p$ in a Riemannian 
manifold --- first examined by Poincar\'{e}~\cite{P05} --- is the 
set of points $q\neq p$ which are endpoints of maximal 
(minimizing) geodesics from $p$. This collection of points is 
more subtle than merely the set of points with multiple geodesics 
to $p$, and in fact, is generally the closure thereof 
(see Klingenberg~\cite[Section 2.1.14]{K95}). 

In the Brownian map this equivalence breaks completely. 
Indeed, 
almost all (in the sense of volume, by
the confluence of geodesics phenomenon and
invariance under re-rooting) 
and most (in the sense of Baire category, by Theorem~\ref{T_frame})
points are the end of a maximal geodesic, and every point is joined by 
multiple geodesics to a dense set of points 
(see the note after the proof of Proposition~\ref{P_mul}).
Moreover, whereas in the Brownian map 
there are points with multiple geodesics to the root which coalesce before 
reaching the root, 
in a Riemannian manifold
any (minimizing) geodesic which is not the unique geodesic between
its endpoints cannot be extended (see, for example, the 
``short-cut principle'' discussed in 
Shiohama, Shioya and Tanaka~\cite[Remark~1.8.1]{SST03}).

We introduce the following notions of cut locus for the 
Brownian map.

\begin{defn}  
  For $x\in \m$, the \emph{weak cut locus} of $x$, denoted $\mul(x)$,
  is the set of points $y\in\m$ with multiple geodesics to $x$. 
  The {\em strong cut locus} of $x$, 
  denoted $\cut(x)$, is the set of points $y\in\m$ to which 
  there are at least two geodesics from $x$ that are disjoint 
  in a neighbourhood of $y$. 
\end{defn}

We will see that for most points $x$, it holds that $\mul(x)=\cut(x)$
(Proposition \ref{P_S=C}).  However, in some sense, $\cut(x)$ is
better-behaved than $\mul(x)$ for the remaining exceptional points,
and we will 
argue in Section~\ref{S_cut-pf} below that $\cut(x)$
is more effective at capturing 
the essence of a cut-locus for the metric space
$(M,d)$.
 
The construction of the Brownian map as a quotient 
of the CRT gives a natural mapping from the CRT to the map. 
Let $\rho$ denote the root of the map. 
Cut-points of the CRT correspond to a dense subset 
$\mul(\rho)\subset \m$ of Hausdorff dimension 2 
(see~\cite{LG10}). Le Gall's description of geodesics reveals that 
$\mul(\rho)$ is almost surely
exactly the set of points with multiple geodesics to  
$\rho$ (see Section~\ref{S_simple}). 
More specifically, for any $y\in\m$, the number of connected 
components of $\mul(\rho)-\{y\}$ is precisely the number of 
geodesics from $y$ to $\rho$. 
This is similar to the case 
of a complete, analytic Riemannian surface homeomorphic 
to the sphere 
(see Poincar\'{e}~\cite{P05} and Myers~\cite{M35}) 
where the cut locus $\mul$ of a point $x$ is a tree 
and the number of ``branches'' emanating from a point in 
$\mul$ is exactly the number of geodesics to $x$.

Since the strong cut locus of the root of the Brownian map 
corresponds to the CRT minus its leaves --- that is, 
almost surely
$\mul(\rho)=\cut(\rho)$, where $\rho$ is the root 
(see Section~\ref{S_simple}) --- it is a fundamental 
subset of the map. 

We obtain analogues of 
Theorems~\ref{T_net-nwd},\ref{T_net-cts}
for the strong cut locus.

\begin{thm}\label{T_cut-nwd}
  Almost surely, for all $x,y\in\m$, $\cut(x)$ and $\cut(y)$ 
  coincide outside a closed, nowhere dense set of zero 
  $\lambda$-measure.
\end{thm}

\begin{thm}\label{T_cut-cts}
  Almost surely, the function $x\mapsto\cut(x)$ is continuous 
  almost everywhere in the following sense.
  
  For $\lam $-almost every $x\in\m$, for any neighbourhood 
  $N$ of $x$, there is a sub-neighbourhood $N'\subset N$ 
  so that 
  $\cut(x')-N$ is the same for all $x'\in N'$.
\end{thm}
 
Theorem~\ref{T_cut-cts} brings to mind the results of 
Buchner~\cite{B77} and Wall~\cite{W77}, which show 
that the cut locus of a fixed point in a compact manifold 
is continuously stable under perturbations of the metric 
on an open, dense subset of its Riemannian 
metrics (endowed with the Whitney topology).

As for the geodesic nets in Theorem~\ref{T_frame}, 
we show that the union of all strong cut loci is a 
small subset of the map.
 
\begin{thm}\label{T_locus}
  Almost surely, $\bigcup_{x\in\m}\cut(x)$ 
  is of first Baire category. 
\end{thm}  

We remark that Gruber~\cite{G82} 
(see also Zamfirescu~\cite{Z91}) shows that for most 
(in the sense of Baire category) convex 
surfaces $X$, for any point $x\in X$, the set of points with 
multiple geodesics to $x$ is of first Baire category. 
Since for typical points $x\in\m$, $\cut(x)$ is exactly 
the set of points with multiple geodesics to $x$ (that is, $\cut(x)=S(x)$, 
see Proposition~\ref{P_S=C}),  
Theorem~\ref{T_locus} shows that this 
property holds almost surely for
almost every point of the Brownian map.
That being said, there is a dense set of atypical points 
$D$ such that every $x\in D$ is connected to \emph{all} 
points outside 
a small neighbourhood of $x$ by multiple geodesics 
(see Proposition~\ref{P_mul}).

\subsection{Geodesic networks}\label{S_networks}

Next, we investigate the structure of geodesic segments 
between pairs of 
points in the Brownian map.

\begin{defn}
  For $x,y\in\m$, the \emph{geodesic network} between $x$ 
  and $y$, denoted $\geo(x,y)$, is the set of points in 
  some 
  geodesic segment between $x$ and $y$.
\end{defn}

Geodesic networks with one endpoint being the root of 
the map (or a typical point by invariance under re-rooting) are 
well understood. As discussed in Section~\ref{S_cut}, for any 
$y\in\m$, the number of connected components in 
$\mul(\rho)-\{y\}$ gives the number of geodesics from $y$ to 
$\rho$. Hence, by properties of the 
CRT, almost surely there is a dense 
set with Hausdorff dimension 2
of points with exactly two geodesics to the root; 
a dense, countable set of points with 
exactly three geodesics to the root; 
and no points connected to the root 
by more than three geodesics. 
By invariance under re-rooting, it follows 
that the set of pairs that are joined by multiple geodesics is 
a zero-volume subset of $(\m^2,\lam\otimes\lam)$ 
(see also Miermont~\cite{M09}). Hence the vast majority 
of networks in the Brownian map consist of a single 
geodesic segment. Furthermore, by Le Gall's description 
of geodesics to the root and invariance under re-rooting, 
geodesic segments 
from a typical point of the Brownian map have 
a specific topological structure.

For $x\in\m$, let $B(x,\eps)$ denote the open ball of radius $\eps$
centred at $x$.

\begin{defn}
  We say that the ordered pair of distinct points $(x,y)$
  is {\em regular} if any two distinct geodesic segments 
  between $x$ and $y$ are disjoint inside, and coincide outside, 
  a punctured ball centred at $y$ of radius less than
  $d(x,y)$. Formally, if $\g$ and $\g'$ are geodesic segments between $x$
  and $y$, then there exists $r\in (0,d(x,y))$ such that $\g\cap
  \g'\cap B(y,r)=\{y\}$ and $\g- B(y,r)=\g'-B(y,r)$.
\end{defn}

For typical points $x$, all pairs $(x,y)$ are regular
(see Section~\ref{S_simple}). 

We note that this notion is not symmetric, that is, $(x,y)$ being
regular does not imply that $(y,x)$ is regular.  In fact, observe that
$(x,y)$ and $(y,x)$ are regular if and only if there is a unique
geodesic from $x$ to $y$.

A key property is the following. 

  \begin{lem}
   \label{sec:geodesic-networks}
   If $(x,y)$ is regular and $\g$ is a geodesic segment between $x$
   and $y$, then for any point $z$ in the relative interior of $\g$,
   the segment $[x,z]\subset \g$ is the unique geodesic segment
   between $x$ and $z$. Hence, any points $z\neq z'$ in the 
   relative interior of $\g$ are joined by a unique geodesic. 

   Consequently, any geodesic segment $\g'$ to $x$ that intersects
   the relative interior of $\g$ at some point $z$ coalesces with $\g$
   from that point on, that is, $\g\cap B(x,d(x,z))=\g'\cap
   B(x,d(x,z))$.  
  \end{lem}

  \begin{proof}
    Let $(x,y)$ be regular and let $\g$ be a geodesic segment between
    $x$ and $y$. Assume that there are two distinct geodesic segments
    $\g_1,\g_2$ between $z$ and $x$, where $z$ is some point in the
    relative interior of $\g$. By adding the sub-segment $[y,z]\subset
    \g$ to $\g_1$ and $\g_2$, we obtain two distinct geodesic segments
    between $y$ and $x$ that coincide in the non-empty neighbourhood
    $B(y,d(y,z))$ of $y$, contradicting the definition of regularity
    for $(x,y)$. This gives the first part of the statement, and the
    second part is a straightforward consequence.
  \end{proof}

We find that all except very few
geodesic networks in the Brownian 
map are, in the following sense, a concatenation of two 
regular networks.

\begin{defn}
  For $(x,y)\in\m^2$ and $j,k\in\NN$, we say that 
  $(x,y)$ induces a \emph{normal $(j,k)$-network}, and 
  write $(x,y)\in\net(j,k)$, if for some $z$ in 
  the relative interior of 
  all geodesic 
  segments between $x$ and $y$,  
  $(z,x)$ and $(z,y)$ are regular and $z$ is connected
  to $x$ and $y$ by exactly $j$ and $k$ geodesic 
  segments, respectively.
\end{defn}

\begin{figure}[h]
\centering
\includegraphics[scale=0.8]{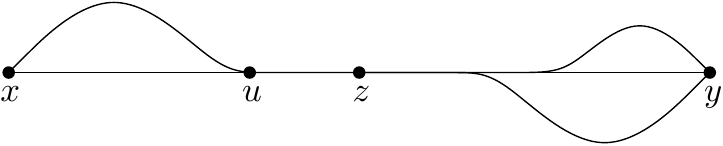}
\caption{
  As depicted, $(x,y)\in\net(2,3)$. 
  Note that $(u,x)$ does not induce 
  a normal $(j,k)$-network. 
}
\label{F_normal}
\end{figure}

In particular, note if $x,y$ are joined by exactly
$k$ geodesics and $(x,y)$ is regular, 
then $(x,y)\in\net(1,k)$. (Take $z$ to be a point
in the relative interior of the geodesic segment
contained in all $k$ segments from $x$ to $y$.)

Not all networks are normal $(j,k)$-networks. For 
instance, if $(x,y)\in\net(j,k)$ and $j>1$, then there is 
a point $u\in\geo(x,y)$ so that $u$ is joined to $x$ by 
two geodesics with disjoint relative interiors. See 
Figure~\ref{F_normal}. That being said, most pairs 
induce normal $(j,k)$-networks. Moreover, for each 
$j,k\in\{1,2,3\}$, there are many normal $(j,k)$-networks 
in the map. Hence, in particular, we establish the existence 
of atypical networks comprised of more than three 
geodesics (and up to nine).

\begin{thm}\label{T_normal}
  The following hold almost surely.
  \begin{enumerate}[nolistsep,label={(\roman*)}]
  \item For any $j,k\in\{1,2,3\}$, $\net(j,k)$ is dense in $\m^2$.
  \item $\m^2- \bigcup_{j,k\in\{1,2,3\}} N(j,k)$ 
  is nowhere dense in $\m^2$.
  \end{enumerate}
\end{thm}

By Theorem~\ref{T_normal}, there are essentially only six types 
of geodesic networks which are dense in the Brownian map. See 
Figure~\ref{F_networks}. 

\begin{figure}[h]
\centering
\includegraphics[scale=0.8]{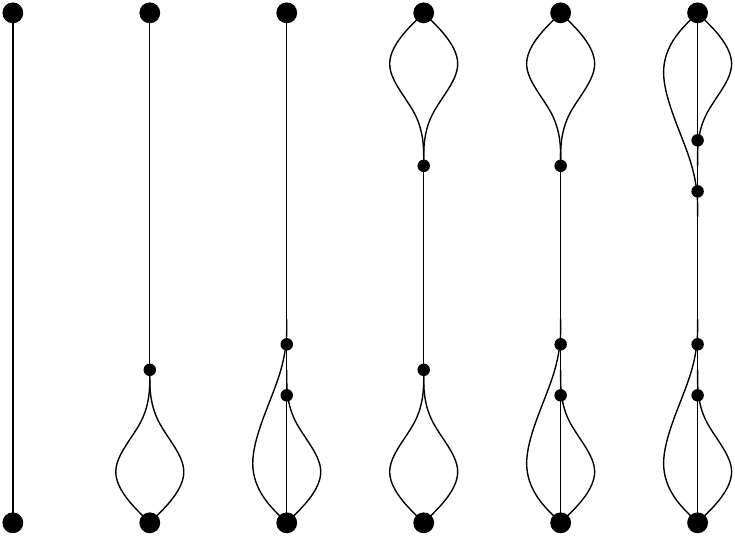}
\caption{Theorem~\ref{T_normal}: 
  Classification of networks which are dense in the Brownian map
  (up to symmetries and homeomorphisms of the sphere).
  } 
\label{F_networks}
\end{figure} 

Since the geodesic net of the root,
or a typical point by invariance under re-rooting, 
is a binary tree --- which follows by the uniqueness of 
local minima of the label process $Z$, see
\cite[Lemma~3.1]{LGP08}, and since  
$\geo(\rho)$ is the tree $[0,1]/\{d_Z=0\}$, 
see Section~\ref{S_simple} --- it can be shown 
using ideas in the proof of 
Theorem~\ref{T_normal-dim} below
that the pairs of small dots near the large dots in the 
3rd, 5th and 6th networks in Figure~\ref{F_networks}  
are indeed distinct points. (That is, Theorem~\ref{T_normal}
would still hold if we were to further require that normal networks
have this additional property.)
For instance, in Figure~\ref{F_normal-dim} below, 
note that 
all geodesic segments from $y$ to $y'$ are 
sub-segments of geodesics from $y$ to the typical
point $z_n$, and hence do not coalesce at the same point.
We omit further discussion on this 
small detail.

It remains an interesting open problem to 
fully classify the types of geodesic networks 
in the Brownian map.

Additionally, we obtain the dimension of the sets 
$N(j,k)$, $j,k\le3$.

For a set $A\subset\m$, let $\H A$ and $\pack A$
denote its Hausdorff and packing dimensions, 
respectively.

\begin{thm}\label{T_normal-dim}
  Almost surely, we have that 
  $\H\net(j,k)= \pack\net(j,k)=2(6-j-k)$, 
  for all $j,k\in\{1,2,3\}$. 
  Moreover, $N(3,3)$ is countable.
\end{thm}

We remark that since $\net(j,k)$, for any 
$j,k\in\{1,2,3\}$, is dense in $\m^2$ 
(by Theorem~\ref{T_normal})
its Minkowski 
dimension is that of $\m^2$, which by 
Proposition~\ref{P_Dim-U} below is 
almost surely equal to 8.

\begin{defn}
  For each $k\in \NN$, let $\pair(k)\subset\m^2$ denote the
  set of pairs of points that are connected by exactly $k$ geodesics.
\end{defn}

Theorems~\ref{T_normal},\ref{T_normal-dim} imply the following results.

\begin{cor}\label{T_pairs}
  Put $K=\{1,2,3,4,6,9\}$. 
  The following hold almost surely.
  \begin{enumerate}[nolistsep,label={(\roman*)}]
  \item For each $k\in K$, $\pair(k)$ is dense in $\m^2$.
  \item $\m^2-\bigcup_{k\in K}\pair(k)$ is nowhere 
  	dense in $\m^2$.
  \end{enumerate}
\end{cor}

\begin{cor}\label{T_pairs-dim}
  Almost surely, we have that $\H\pair(2)\ge6$, 
  $\H\pair(3)\ge4$, $\H\pair(4)\ge4$ 
  and $\H\pair(6)\ge2$.
\end{cor}

We expect the lower bounds in Corollary~\ref{T_pairs-dim} 
to give the correct Hausdorff dimensions of the sets 
$P(k)$, $k\in K-\{1,9\}$. As discussed in Section~\ref{S_cut}, 
$\pair(1)$ is of full volume, and hence $\H\pair(1)=8$. 
We suspect that $\pair(9)$ is countable. It would be of 
interest to determine if the set $\pair(k)$ is non-empty 
for some $k\notin K$, and whether there is any $k\not\in K$ 
for which it has positive dimension. 
We hope to address these issues in future work.

\subsection{Confluence points}\label{S_conf-point}

Our key tool is a strengthening of the confluence of 
geodesics phenomenon of Le Gall~\cite{LG10} 
(see Section~\ref{S_at}). We find that for any 
neighbourhood $N$ of a typical point in the Brownian 
map, there is a \emph{confluence point} $x_0$ between 
a sub-neighbourhood $N'\subset N$ and the complement 
of $N$. See Figure~\ref{F_conf-point}. 

\begin{pro}\label{P_conf-point}
  Almost surely, for $\lam$-almost every $x\in\m$, 
  the following holds. For any neighbourhood $N$ of $x$, 
  there is a sub-neighbourhood $N'\subset N$ and 
  some $x_0\in N-N'$ so that all geodesics between 
  any points $x'\in N'$ and $y\in N^c$ pass through $x_0$.
\end{pro}

\begin{figure}[h]
\centering
\includegraphics[scale=0.8]{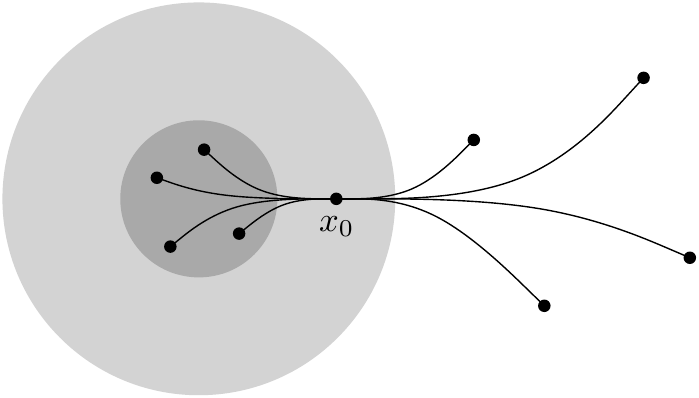}
\caption{Proposition~\ref{P_conf-point}:
  All geodesics from points in 
  $N'$ to points in the complement of $N\supset N'$
  pass through a confluence point $x_0$.
}
\label{F_conf-point}
\end{figure}

\begin{defn}
  We say that a sequence of geodesic segments 
  $\g_n$ converges to a geodesic segment $\g$, 
  and write $\g_n\to \g$, if $\g_n$ converges to $\g$ 
  with respect to the Hausdorff topology.
\end{defn}

Since $(\m,\d)$ is almost surely homeomorphic 
to $\SS^2$, and hence almost surely compact, 
the following lemma is a straightforward consequence 
of the Arzel\`{a}-Ascoli Theorem
(see, for example, 
Bridson and Haefliger {\cite[Corollary 3.11]{BH99}}). 

\begin{lem}\label{L_AA}
  Almost surely, the set of geodesic segments in 
  $(\m,\d)$ is compact (with respect to the 
  Hausdorff topology).
\end{lem}

Our key result, Proposition~\ref{P_conf-point}, is 
related to the fact that many sequences of geodesic
segments in the Brownian map converge in a 
stronger sense.

\begin{defn}
  We say that a sequence of geodesic segments $[x_n,y_n]$ \emph{converges
    strongly} to $[x,y]$, and write $[x_n,y_n] \tto [x,y]$, if $x_n\to x$,
  $y_n\to y$, and for any geodesic segment $[x',y'] \subset (x,y)$
  (excluding the endpoints) we have that $[x',y'] \subset [x_n,y_n]$ for
  all sufficiently large $n$.
\end{defn}

Strong convergence is stronger than convergence in the Hausdorff topology.
Indeed, if $x',y'$ are $\eps$ away from $x,y$ along $[x,y]$, then for large
$n$ $[x',y']\subset [x_n,y_n]$.  Moreover, since
$d(x_n,x') \leq d(x_n,x) + \eps$ for all such $n$, $[x_n,x']$ is eventually contained in
$B(x,2\eps)$.  Similarly, $[y',y_n]$ is eventually contained in
$B(y,2\eps)$.  In the Euclidean plane, or generic smooth manifolds, strong
convergence does not occur.  In contrast, in the Brownian map it is the
norm, as we shall see below. In light of this we also make the following
definition.

\begin{defn}
  A geodesic segment $\gamma$ is called a 
  \emph{stable geodesic} if whenever
  $[x_n,y_n]\to \gamma$ we also have $[x_n,y_n] \tto \gamma$.  Otherwise,
  $\gamma$ is called a \emph{ghost geodesic}.
\end{defn}

\begin{pro}\label{P_key}
  Almost surely, for $\lam$-almost every $x\in\m$, for all $y\in\m$,
  all sub-segments of all geodesic segments $[x,y]$ are stable. 
\end{pro}
 
Proposition~\ref{P_conf-point} follows by 
Proposition~\ref{P_key}, the confluence of geodesics 
phenomenon, and the fact that $(\m,\d)$ is almost 
surely compact (see Section~\ref{S_near}).

In closing, we remark that it would be interesting to know if
Proposition~\ref{P_key} holds for all $x\in\m$, that is, are all geodesics
in $\m$ stable, or are there any ghost geodesics?  Ghost geodesics have
various properties, and in particular they intersect every other geodesic
in at most one point.  It would be quite surprising if such geodesics
exist, and we hope to rule them out in future work.  We thus expect an
analogue of Proposition~\ref{P_conf-point} to hold for all $x\in\m$. If so,
then as a consequence, we would obtain the following result.

\begin{con}
  Almost surely, the geodesic framework of the Brownian map, $F\subset\m$,
  is of Hausdorff dimension 1.
\end{con}

In this way, we suspect that although the Brownian 
map is a complicated object of Hausdorff dimension 4,
it has a relatively simple geodesic framework
which is of first Baire category (Theorem~\ref{T_frame}) 
and Hausdorff dimension 1.

\section{Preliminaries}\label{S_prelims}

In this section, we briefly recount the construction 
of the Brownian map and what is known regarding its geodesics.

\subsection{The Brownian map}\label{S_BM}

Fix $q\in\{3\}\cup2(\NN+1)$ and set $c_q$ equal to 
$6^{1/4}$ if $q=3$ or $(9/q(q-2))^{1/4}$ if $q>3$. 
Let $M_n$ denote a uniform $q$-angulation of the sphere 
(see Le Gall and Miermont~\cite{LGM12}) with $n$ faces, and 
$d_{n}$ the graph distance on $M_n$ scaled by $c_q n^{-1/4}$.
The works of Le Gall~\cite{LG13} and Miermont~\cite{M13} 
(for $q=4$) show that in the Gromov-Hausdorff topology on 
isometry classes of compact metric spaces 
(see Burago, Burago and Ivanov~\cite{BBI01}), 
$(M_n,d_n)$ converges in distribution to a random 
metric space called the \emph{Brownian map} $(\m,\d)$.

The Brownian map has also been identified as the 
  scaling limit of 
  several other types of maps, 
  see~\cite{A16,ABA13,BLG13,BJM13,LG13}.

The construction of the Brownian map involves a normalized 
Brownian excursion $\ee=\{\ee_t:t\in[0,1]\}$, a random 
$\RR$-tree $(\Te,d_{\ee})$ indexed by $\ee$, and a Brownian 
label process $Z=\{ Z_a: a\in \Te \}$. More specifically, define 
$\Te=[0,1]/\{\dd=0\}$ as the quotient under the pseudo-distance 
\[
\dd(s,t)
=\ee_s+\ee_t-2\cdot\min_{s\wedge t\le u\le s\vee t}\ee_u,
\quad s,t\in[0,1]
\]
and equip it with the quotient distance, again denoted by $\dd$. 
The random metric space $(\Te,\dd)$ is Aldous' 
\emph{continuum random tree}, or CRT. Let $\pp:[0,1]\to\Te$ 
denote the canonical projection. Conditionally given $\ee$, 
$Z$ is a centred Gaussian process satisfying 
${\bf E}[(Z_s-Z_t)^2]=\dd(s,t)$ for all $s,t \in [0,1]$. 
The random process $Z$ is the so-called 
\emph{head of the Brownian snake} (see~\cite{LGM12}). 
Note that $Z$ is constant on each equivalence class 
$\pp^{-1}(a)$, $a\in\Te$. In this sense, $Z$ is Brownian 
motion indexed by the CRT.

Analogously to the definition of $d_\ee$, we put
\[
d_Z(s,t)
=Z_s+Z_t-2\cdot
\max\left\{ \inf_{u\in[s,t]}Z_u, \inf_{u\in[t,s]}Z_u \right\},
\quad s,t\in[0,1]
\]
where we set $[s,t]=[0,t]\cup[s,1]$ in the case that $s>t$.
Then, to obtain a pseudo-distance on $[0,1]$, we define
\[
D^*(s,t)
=\inf\left\{\sum_{i=1}^k d_Z(s_i,t_i) : 
s_1=s,t_k=t, d_\ee(t_i,s_{i+1})=0\right\}, 
\quad s,t\in[0,1].
\]
 
Finally, we set $\m=[0,1]/\{D^*=0\}$ and endow it with 
the quotient distance induced by $D^*$, which we denote 
by $d$. An easy property of the Brownian map is that $d_\ee(s,t)=0$
implies $D^*(s,t)=0$, so that $\m$ can also be seen as a quotient of
$\Te$, and we let $\Pi:\Te\to \m$ denote the canonical projection, 
and put ${\bf p}=\Pi\circ \pp$. Almost surely, the process $Z$ 
attains a unique minimum on $[0,1]$, say at $t_*$. We set 
$\rho={\bf p}(t_*)$. The random metric space 
$(\m,\d)=(\m,\d,\rho)$ is called the \emph{Brownian map} 
and we call $\rho$ its \emph{root}. Being the 
Gromov-Hausdorff limit of geodesic spaces, 
$(\m,\d)$ is almost surely a geodesic space 
(see~\cite{BBI01}).

Almost surely, for every pair of distinct points $s\neq t\in[0,1]$, at
most one of $d_\ee(s,t)=0$ or $d_Z(s,t)=0$ holds, except in the
particular case $\{s,t\}=\{0,1\}$ where both identities hold
simultaneously (see~\cite[Lemma 3.2]{LGP08}).  Hence, only leaves
(that is, non-cut-points) of $\Te$ are identified in the construction
of the Brownian map; and this occurs if and only if they have the same
label and along either the clockwise or counter-clockwise,
contour-ordered path around $\Te$ between them, one only finds
vertices of larger label.  Thus, as mentioned at the beginning of
Section~\ref{S_intro}, in the construction of the Brownian map,
$(\Te,Z)$ is a continuum analogue for a well-labelled plane tree, and
the quotient by $\{D^*=0\}$ for the CVS-bijection (which, as discussed in 
Section~\ref{S_intro}, 
identifies well-labelled plane trees with rooted planar maps).  See
Section~\ref{S_simple} for more details.

Lastly, we note that although the Brownian map is a rooted 
metric space, it is not so dependent on its root. The volume 
measure $\lam $ on $\m$ is defined as the push-forward of 
Lebesgue measure on $[0,1]$ via $\PP$. Le Gall~\cite{LG10} 
shows that the Brownian map is \emph{invariant under re-rooting} 
in the sense that if $U$ is uniformly distributed over $[0,1]$ and 
independent of $(\m,\d)$, then $(\m,\d,\rho)$ and $(\m,\d,\PP(U))$ 
are equal in law. Hence, to some extent, the root of the 
map is but an artifact of its construction.

\subsection{Simple geodesics}\label{S_simple}

Recall that a \emph{corner} of a vertex $v$ in a discrete 
plane tree $T$ is a sector centred at $v$ and 
delimited by edges which precede and follow $v$ 
along a contour-ordered path around $T$.
Leaves of a tree have exactly one corner, and in general,
the number of corners of $v$ is equal to the number
of connected components in $T-\{v\}$. Similarly, we may 
view the $\RR$-tree $\Te$ as having corners, 
however in this continuum setting all sectors reduce 
to points. Hence, for the purpose of the 
following (informal) discussion, let us think of  
each $t\in[0,1]$ as corresponding to a \emph{corner} of $\Te$ 
with label $Z_t$. 

Put $Z_*=Z_{t_*}$. As it turns out, $d(\rho,\PP(t))=Z_t-Z_*$ for 
all $t\in[0,1]$ (see~\cite{LG07}). In other words, up to a shift by 
the minimum label $Z_*$, the Brownian label of a point in $\Te$ 
is precisely the distance to $\rho$ from the corresponding point 
in the Brownian map.

All geodesics to $\rho$ are \emph{simple geodesics}, 
constructed as follows. For $t\in[0,1]$ and $\ell\in[0,Z_t-Z_*]$, 
let $s_t(\ell)$ denote the point in $[0,1]$ corresponding to the 
first corner with label $Z_t-\ell$ in the clockwise, contour-ordered 
path around $\Te$ beginning at the corner corresponding to $t$. 
For each such $t$, the image of the
function $\Gamma_t:[0,Z_t-Z_*]\to\m$ 
taking $\ell$ to $\PP(s_t(\ell))$ is a geodesic 
segment
from $\PP(t)$ to 
$\rho$. Moreover, the main result of~\cite{LG10} shows that all 
geodesics to $\rho$ are of this form. Hence, the geodesic net 
of the root, $\geo(\rho)$, is precisely the set of cut-points of the 
$\RR$-tree $\TZ=[0,1]/\{d_Z=0\}$ projected into $\m$. 

These results mirror the fact that from each corner of a 
labelled, discrete plane tree, the CVS-bijection draws 
geodesics to the root of the resulting map in such a way 
that the label of a vertex visited by any such geodesic 
equals the distance to the root. 
See~\cite{LG14,LG10} for further details.

Moreover, since the cut-points of $\Te$ are its vertices with 
multiple corners, we see that the set $\mul(\rho)$ 
(discussed in Section~\ref{S_cut}) of points with multiple 
geodesics to $\rho$ is exactly the set of cut-points of the 
$\RR$-tree $\Te=[0,1]/\{\dd=0\}$ projected into $\m$.

Furthermore, since points in $\mul(\rho)$ correspond to 
leaves of $\TZ$ (see~\cite[Lemma 3.2]{LGP08}), geodesics 
to the root of the map (or a typical point, by invariance 
under re-rooting)
have a particular topological structure,
as discussed in Section~\ref{S_networks}.
We state this here for the record.

\begin{pro}\label{P_reg}
  Almost surely, for $\lambda$-almost every $x$,
  for all $y\in\m$, $(x,y)$ is regular. 
\end{pro}

Hence, as mentioned in Section~\ref{S_cut}, we have 
that $\mul(\rho)=\cut(\rho)$. That is, all points with multiple 
geodesics to the root 
are in the strong cut locus of the root.

\subsection{Confluence at the root}\label{S_at}

As discussed in Section~\ref{S_net}, a 
\emph{confluence of geodesics} is observed at the root 
of the Brownian map. Combining this with invariance
under re-rooting, the following result is obtained.

\begin{lem}[{Le Gall~\cite[Corollary 7.7]{LG10}}]\label{L_LG}
  Almost surely, for $\lam $-almost every $x\in\m$, the 
  following holds. For every $\eps>0$ there is an 
  $\eta\in(0,\eps)$ so that if $y,y'\in B(x,\eps)^c$, 
  then any pair of geodesics from $x$ to $y$ and $
  y'$ coincide inside of $B(x,\eta)$.
\end{lem}

Moreover, geodesics to the root of the map 
tend to
coalesce quickly.

For $t\in[0,1]$, let $\g_t$ denote the image of the simple 
geodesic $\Gamma_t$ from $\PP(t)$ to the root of the map 
$\rho$ (see Section~\ref{S_simple}).

\begin{lem}[Miermont {\cite[Lemma 5]{M13}}]\label{L_M}
  Almost surely, for all $s,t\in[0,1]$, $\g_s$ and 
  $\g_t$ coincide outside of $B(\PP(s),d_Z(s,t))$.
\end{lem}

We require the following lemma. 

\begin{lem}\label{L_at}
  Almost surely, for $\lam $-almost every $x\in\m$, the 
  following holds. For any $y\in\m$ and 
  neighbourhood $N$ of $y$, there is a 
  sub-neighbourhood $N'\subset N$ so that if $y'\in N'$, 
  then any geodesic from $x$ to $y'$ coincides with a 
  geodesic from $x$ to $y$ outside of $N$.
\end{lem}

\begin{proof}
  Let $\rho$ denote the root of the map. Let $y\in\m$ 
  and a neighbourhood $N$ of $y$ be given. Select 
  $\eps>0$ so that $B(y,\eps)\subset N$. Let $N_\eps$ 
  denote the set of points $y'\in\m$ with the property that 
  for all $t'\in[0,1]$ for which $\PP(t')=y'$, there exists some 
  $t\in[0,1]$ so that $\PP(t)=y$ and $d_Z(t,t')<\eps$. As 
  discussed in Section~\ref{S_simple}, Le Gall~\cite{LG10} 
  shows that all geodesics to $\rho$ are simple geodesics. 
  Hence, by Lemma~\ref{L_M}, any geodesic from $\rho$ to 
  a point $y'\in N_\eps$ coincides with some geodesic from 
  $\rho$ to $y$ outside of $N$. 
  
  We claim that $N_\eps$ is a 
  neighbourhood of $y$. To see this, note that if 
  $\PP(t_n)=y_n\to y$ in $(\m,\d)$, then there is a
  subsequence $t_{n_k}$ so that 
  for some $t_y\in[0,1]$, we have that
  $t_{n_k}\to t_y$ as $k\to\infty$. 
  Hence $d_Z(t_y,t_{n_k})<\eps$ for all large $k$, and
  since $\PP$ is continuous (see~\cite{LG10}), $\PP(t_y)=y$.
  Therefore, for any $y_n\to y$ in $(\m,\d)$, 
  $y_n\notin N_\eps$ for at most finitely many $n$,
  giving the claim.
  
  Hence the lemma follows by invariance under re-rooting. 
\end{proof}

We remark that the size of $N'$ in Lemma~\ref{L_at} depends strongly
on $x$ and $y$. For instance, for a fixed $\eps>0$ and convergent
sequences of typical points $x_n$ (that is, points satisfying the
statement of Lemma~\ref{L_at}) and general points $y_n$, for each $n$
let $\eta_n>0$ be such that the
statement of the 
 lemma holds for the pair $x_n,y_n$
with $N_n=B(y_n,\eps)$ and $N_n'=B(y_n,\eta_n)$.  It is quite possible
that $\eta_n\to0$ as $n\to\infty$.

\subsection{Dimensions}\label{S_dims}

Finally, we collect some facts about the dimension of various subsets of
the Brownian map.  These statements are easily derived from established
results, but are not explicitly stated in the literature.

For a metric space $X\subset\m$, let $\H X$ denote its Hausdorff dimension,
$\pack X$ its packing dimension, and $\lM X$ (resp.\ $\uM X$) its lower
(resp.\ upper) Minkowski dimension.  If the lower and upper Minkowski
dimensions coincide, we denote their common value by $\M X$.  We note that
for any metric space $X$ we have
\begin{align*}
  \H X\le \lM X\le \uM X && \text{and} && \H X\le \pack X\le \uM X.
\end{align*}
See Mattila~\cite{M95}, for instance, for detailed definitions and other
properties of these dimensions.

We require the following result, which is implicit in 
Le Gall's~\cite{LG07} proof that $\dim\m=4$. 
For completeness, we include a proof via the uniform 
volume estimates of balls in the Brownian map.

\begin{pro}\label{P_Dim-U}
  Almost surely, for any non-empty, open subset
  $U\subset\m$, we have that 
  $\lambda(U)>0$ (hence $\lambda$ has full support) and
  $\H U=\pack U=\M U=4$.
\end{pro}

\begin{proof}
  Let a non-empty, open subset $U\subset\m$ be given.
    Fix some arbitrary $\eta>0$.

  By~\cite[Lemma 15]{M13}, there is a
  $c\in(0,\infty)$ and $\eps_0>0$ so that for all $\eps\in(0,\eps_0)$ and
  $x\in\m$, we have that $\lam(B(x,\eps))\ge c\eps^{4+\eta}$.  
  In particular, $\lambda(U)>0$.
  For $\eps>0$, let $N(\eps)$ denote the number of balls of radius $\eps$
  required to cover $\m$. 
  By a
  standard argument, it follows that 
  there exists a $c'\in(0,\infty)$ so that for all
  $\eps\in(0,2\eps_0)$ we have $N(\eps)\le c'\eps^{-(4+\eta)}$.  It follows
  directly that $\uM \m\le4+\eta$, and the same bound holds for  $U\subset
  M$. 
  
  On the other hand, by~\cite[Lemma 14]{M13} (a consequence
  of~\cite[Corollary 6.2]{LG07}), there is a $C\in(0,\infty)$ so that for
  all $\eps>0$ and $x\in\m$, we have that
  $\lambda(B(x,\eps))\le C \eps^{4-\eta}$.  In particular, for all $\eps>0$
  and $x\in U$ we have $\lambda(B(x,\eps)\cap U)\le C\eps^{4-\eta}$.  It follows
  that $\H U\ge 4-\eta$ (see, for example,
  Falconer~\cite[Exercise~1.8]{F86}).

  Since $\eta>0$ is arbitrary, the general dimension inequalities imply the
  claim.
\end{proof}

\begin{defn}
  For $x\in\m$, and $k\ge 1$ or $k=\infty$, let $\mul_k(x)$ denote the set
  of points $y\in \m$ with exactly $k$ geodesics to $x$.
\end{defn}

We believe that $\mul_\infty(x)$ is empty for all $x$. In fact, it is
plausible that all $\mul_k(x)$ are empty for all $k>k_0$ (perhaps even $k_0=9$).

In particular, the weak cut locus $S(x)$, as defined in Section~\ref{S_cut},
is equal to $S_\infty(x) \cup \bigcup_{k\geq 2}S_k(x)$. As discussed in
Section~\ref{S_networks}, by Le Gall's description of geodesics to the
root, properties of the CRT, and invariance under re-rooting, we have the
following result.

\begin{pro}\label{P_S123}
  Almost surely, for $\lambda$-almost every $x\in\m$ 
  \begin{enumerate}[nolistsep,label=(\roman*)]
  \item $S(x) = S_2(x) \cup S_3(x)$;
  \item $S_2(x)$ is dense, and has Hausdorff dimension 2 (and measure 0);
  \item $S_3(x)$ is dense and countable.
  \end{enumerate}
\end{pro}

We observe that the proof 
in \cite[Proposition 3.3]{LG10} 
that 
$S(\rho)$ is almost surely of Hausdorff dimension 2
gives additional information.

\begin{pro}\label{P_ScapU}
  Almost surely, for $\lambda$-almost every $x\in\m$, 
  for any non-empty, open set $U\subset\m$ and each 
  $k\in\{1,2,3\}$, we have that
  \[
    \dim(\mul_k(x)\cap U) = \pack(\mul_k(x)\cap U) = 2(3-k).
  \]
\end{pro}

\begin{proof}
  By invariance under re-rooting, it suffices to prove 
  the claim holds almost surely when $x=\rho$ is the root of the map.
  
  Let a non-empty, open subset $U\subset\m$ be given.
  
  Let $S=S(x)$ and $S_i=S_i(x)$ for $i=1,2,3$.  
  By Proposition~\ref{P_S123}(i), 
  $S=S_2\cup S_3$ and $M-\{x\}=S_1\cup S$. 
  
  First, we note that by Proposition~\ref{P_S123}(iii), 
  $S_3\cap U$ is countable, and so has Hausdorff and packing dimension 0.
  
  From~\cite{LG10}  we have that $S$
  is the image of the cut-points (or skeleton) of the CRT, 
  ${\rm Sk}\subset \Te$, under the projection $\Pi:\Te\to\m$.
  Moreover, $\Pi$ is H{\"o}lder continuous
  with exponent $1/2-\eps$
  for any $\eps>0$, and restricted to ${\rm Sk}$,
  $\Pi$ is a homeomorphism from ${\rm Sk}$ onto $S$. 
  
  Note that ${\rm Sk}$ is of packing dimension 1, being
  the countable union of sets which are isometric to line
  segments (recall that the packing dimension of a 
  countable union of sets is the supremum of
  the dimension of the sets). Hence, by the H{\" o}lder 
  continuity of $\Pi$, it follows that $\pack \mul\le2$
  (see, for instance, \cite[Exercise 6, p.\ 108]{M95})
  and so in particular, we find that $\pack(S\cap U)\le2$.
  
  On the other hand, by the density of $S$ in $\m$
  and since $\Pi$ is a homeomorphism from ${\rm Sk}$
  to $S$, 
  we see that there is a geodesic segment in ${\rm Sk}$ that is projected
  to a path in $S\cap U$.
  In the proof of~\cite[Proposition 3.3]{LG10} 
  it is shown that the Hausdorff dimension 
  of any such path 
  is at least 2. Hence $\H(S\cap U)\ge2$.  
  
  Altogether, by the general dimension inequality 
  $\H A\le \pack A$, 
  we find that $S\cap U$ has
  Hausdorff and packing dimension 2. 
  
  Therefore, since $S_3\cap U$ has
  Hausdorff and packing dimension 0 and $S=S_2\cup S_3$, 
  it follows that
  $S_2\cap U$ has Hausdorff and packing dimension 2.
  Moreover, since by Proposition~\ref{P_Dim-U}, $U$
  has Hausdorff and packing dimension 4 
  and $M-\{x\}=S_1\cup S$,
  we find that 
  $S_1\cap U$ has Hausdorff and packing dimension 4.
\end{proof}

In closing, we note that 
Propositions~\ref{P_S123},\ref{P_ScapU} 
imply the following result. 

\begin{pro}\label{P_Dim-S}
  Almost surely, for $\lambda$-almost every $x\in\m$,
  $S(x)$ is dense, $\H S(x) = \pack S(x) = 2$, and 
  $\lam(S(x))=0$.
\end{pro}

\section{Confluence near the root}\label{S_near}

We show that a confluence of geodesics is observed near 
the root of the Brownian map, strengthening the results 
discussed in Section~\ref{S_at}. Specifically, we establish 
the following result.

\begin{lem}\label{L_near}
  Almost surely, for $\lam $-almost every $x\in\m$, the following holds. 
  For any $y\in\m$ and neighbourhoods $N_x$ of $x$ and $N_y$ of $y$, 
  there are sub-neighbourhoods $N_x'$ and $N_y'$ so that if $x'\in N_x'$ 
  and $y'\in N_y'$, then any geodesic segment
  from $x'$ to $y'$ coincides with some
  geodesic segment
  from $x$ to $y$ outside of $N_x\cup N_y$.
\end{lem}

We note that Lemma~\ref{L_near} strengthens Lemma~\ref{L_at} in that 
it allows for perturbations of both endpoints of a geodesic.

Once Lemma~\ref{L_near} is established, our key result follows easily 
by Lemma~\ref{L_LG} and the fact that the Brownian map is almost surely 
compact. 

\begin{proof}[Proof of Proposition~\ref{P_conf-point}]
  By invariance under re-rooting, it suffices to prove the claim when $x=\rho$
  is the root of the map.  Let an (open) neighbourhood $N$ of $x$ be given.
  By Lemma~\ref{L_LG}, there is a point $x_0\in N-\{x\}$ which is
  contained in all geodesic segments between $x$ and points $y\in
  N^c$. Hence, by Lemma~\ref{L_near}, for each $y\in N^c$ there is an
  $\eta_y>0$ so that $x_0$ is contained in all geodesic segments between
  points $x'\in B(x,\eta_y)$ and $y'\in B(y,\eta_y)$.  Since $N^c$ is
  compact, it can be covered by finitely many balls $B(y,\eta_y)$, say with
  $y\in Y$.  Put $N'=B(x,\min_{y\in Y}\eta_y)$. If $y_0\in N^c$, then
  $y_0\in B(y,\eta_y)$ for some $y\in Y$, and thus all geodesics from
  points $x'\in N'\subset B(x,\eta_y)$ to $y_0$ pass through $x_0$.
\end{proof}

The rest of this section contains the proof of Lemma~\ref{L_near}.  By
invariance under re-rooting, we may and will assume that $x$ is in
fact the root of the Brownian map.  In rough terms, we must rule out
the existence of a sequence of geodesic segments $[x_n,y_n]$
converging to a geodesic segment $[x,y]$, but not converging
strongly in the sense given in Section \ref{S_conf-point}.

For the remainder of this section we fix a realization of the Brownian map 
exhibiting the almost sure properties of the random metric space
$(\m,\d)$ that will be required below, notably the fact that $\m$ is
homeomorphic to the $2$-dimensional sphere.  
Slightly abusing notation, let us refer to this realization as $(\m,\d)$. 
We also fix a point $y\neq x \in\m$ and a geodesic segment
$\g=[x,y]$ between $x$ and $y$.

We utilize a dense subset $T\subset\m$ of points, which we
refer to as \emph{typical points,} containing the root $x$, and such that
\begin{enumerate}[nolistsep,label={(\roman*)}]
\item the claims of Proposition~\ref{P_reg} 
and Lemma~\ref{L_at} hold for all $u\in T$;
\item for each $u,v\in T$, there is a unique geodesic
  	from $u$ to $v$.
\end{enumerate}

Such a set exists almost surely.  For example, the set of equivalence
classes containing rational points almost surely works.  
We may assume that
$T$ exists for the particular realization of $(\m,\d)$ we have selected.
It is in fact possible to choose $T$ to have full $\lambda$-measure,
but for now, we only need it to be dense in $M$.

In what follows, we will at times shift our attention to the homeomorphic 
image of a neighbourhood of $\g$ in which our arguments are more 
transparent. Whenever doing so, we will appeal only to topological 
properties of the map. We let $d_{\rm E}$ be the 
Euclidean distance on $\CC$, and for $w\in 
\CC$ and $r>0$, we let $B_{\rm E}(w,r)$ be the open Euclidean ball 
centered at $w$ with radius $r$.

Fix a homeomorphism $\tau$ from $\m$ to $\hat\CC$.
The image of $\g$ under $\tau$ is a simple arc in $\hat\CC$. Let
$\phi$ be a homeomorphism from this arc onto the unit interval
$I=[0,1]\subset\RR\subset\CC$, with $\phi(\tau(x))=0$ and thus
$\phi(\tau(y))=1$.  By a variation of the Jordan-Sch\"{o}nflies
Theorem (see Mohar and Thomassen~\cite[Theorem 2.2.6]{MT01}), $\phi$
can be extended to a homeomorphism from
$\hat\CC$ onto $\hat\CC$.  Hence $\phi\circ\tau|_{\gamma}$ can be
extended to a homeomorphism from $\m$ to $\hat\CC$ sending $\g$ onto
$I$.  We fix such a homeomorphism, and denote it by $\psi$.

Since $M$ is homeomorphic to $\hat\CC$, 
once the geodesic $\gamma$ is fixed
we can think of the Brownian map as just 
$\hat\CC$ with a random metric (for
which $[0,1]$ is a geodesic).  The
reader may well do this, and then $\psi$ becomes the identity.  
We do not
take this route, since that would require showing that $\psi$ can be
constructed in a measurable way, which we prefer to avoid.

\begin{defn}
  Let $\HH_+=\{w\in\CC:{\rm Im}\, w>0\}$ 
  (resp.\ $\HH_-=\{w\in\CC:\mathrm{Im}\,w<0\}$) denote the open 
  upper (resp.\ lower) half-plane of $\CC$. We refer to 
  $L=\psi^{-1}(\HH_+)$ 
  (resp.\ $R = \psi^{-1}(\HH_-)$) 
  as the \emph{left} (resp.\ \emph{right}) \emph{side of $\g$}. 
\end{defn}

\begin{lem}\label{L_u}
  Let $u,v\in\g$. For all $\delta>0$, there are typical points 
  $u_\ell\in B(u,\delta)\cap L\cap T$ and 
  $v_\ell\in B(v,\delta)\cap L\cap T$ so that 
  $[u_\ell,v_\ell]-\g$ is contained in $(B(u,\delta)\cup B(v,\delta))\cap L$.
  (See Figure~\ref{F_u}.)
  An analogous statement holds replacing $L$ with $R$. 
\end{lem}

\begin{figure}[h]
  \centering
  \includegraphics[scale=0.8]{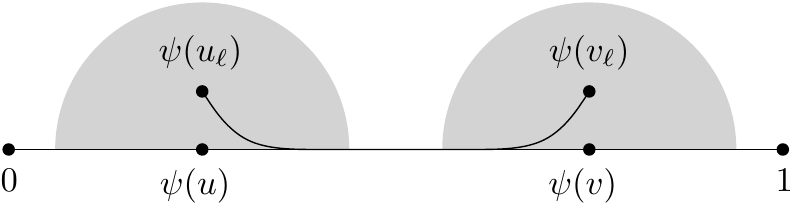}
  \caption{Lemma \ref{L_u}: $[u_\ell,v_\ell]-\g$ is contained in
    $(B(u,\delta)\cup B(v,\delta)) \cap L$ (as viewed
    through the homeomorphism $\psi$).}
  \label{F_u}
\end{figure}

\begin{proof}
  Let $\delta>0$ and $u,v\in\g$ be given.  We only discuss the argument for
  the left side of $\g$, since the two cases are symmetrical.  Moreover, we
  may assume that $u,v,x,y$ are all distinct.  Indeed, suppose the lemma
  holds with distinct $u,v,x,y$.  If we shift $u,v$ along $\gamma$ by at
  most $\eta>0$ and apply the lemma with $\delta'=\delta-\eta$, the
  resulting $u_\ell,v_\ell$ will satisfy the requirements of the lemma
  for $u,v$ and $\delta$.  
  Without loss of generality, we 
  further assume $x,u,v,y$
  appear on $\gamma$ in that order.

  We may and will assume that $\delta < d(u,x) \wedge d(v,y)$. In
  particular, $B(u,\delta)$ and $B(v,\delta)$ do not contain the
  extremities $x,y$ of $\g$.  Let $\delta'>0$ be small enough so that
  $B_{\rm E}(\psi(v),\delta')\subset \psi(B(v,\delta))$.
  Note that the
  Euclidean ball $B_{\rm E}(\psi(v),\delta')$ does not contain 
  $0,1\in \CC$, and so $N=\psi^{-1}(B_{\rm E}(\psi(v),\delta'))$ does not
  intersect the extremities $x,y$ of $\g$.

  Let us apply Lemma \ref{L_at} to the points $x,v$ (using the fact that
  $x$ is typical) and the neighbourhood 
  $N=\psi^{-1}(B_{\rm E}(\psi(v),\delta'))$ of $v$ defined above.  
  According to this lemma,
  there exists a neighbourhood $N'\subset N$ of $v$ such that any geodesic
  segment $\g'$ between a point $v'\in N'$ and $x$ coincides with some
  geodesic between $v$ and $x$ outside $N$.  Since $x,y\notin N$, $\gamma'$
  must first encounter $\g$ (if we see $\g'$ as parameterized from $v'$ to
  $x$) at a point $w$ in the relative interior of $\g$.  Since $(x,y)$ is
  regular, we apply Lemma~\ref{sec:geodesic-networks} to conclude that $\g$
  and $\g'$ coincide between $w$ and $x$ and are disjoint elsewhere.

  If we further assume that $v'\in N' \cap L$ is in the left side
  of $\g$, then we claim that the sub-arc $[v',w)\subset\g'$ is
  contained in $L$.  Indeed, $\psi([v',w))$ is contained in the Euclidean
  ball $B_{\rm E}(\psi(v),\delta')$, starts in $\HH_+$, and is disjoint of $I$, 
  and so, it is contained in the upper half of the ball.

  Since $T$ is dense in $M$, we can take some typical $v_\ell \in N'\cap
  L\cap T$.  For this choice, the geodesic segment $[x,v_\ell]$ is
  unique, and $[x,v_\ell]-\gamma$ is included in $B(v,\delta) \cap L$.

  Assume also $\delta < \frac12 d(u,v)$.  By a similar argument, in which
  $v_\ell$ assumes the role of $x$ (which is a valid assumption since
  $v_\ell\in T$), for any $u'$ close enough to $u$, any 
  geodesic
  $[u',v_\ell]$ coalesces with $[x,v_\ell]$ within $B(u,\delta)$.  
  Taking
  such a $u'=u_\ell$ in $T\cap L$, we get that
  $[v_\ell,u_\ell]- [v_\ell,x]\subset B(u,\delta)\cap L$, and hence
  $[u_\ell,v_\ell]- \g\subset (B(u,\delta)\cup B(v,\delta))\cap L$, as required.
\end{proof}

In the next lemma, recall the two notions of
convergence (standard and strong) of geodesic segments given in Section
\ref{S_conf-point}. 

\begin{lem}\label{L_parallel}
  Suppose that $[x',y']\subset\g$ and $[x_n,y_n]\to[x',y']$ 
  as $n\to\infty$. Then
  we have the strong convergence
  $[x_n,y_n] \tto [x',y']$. 
\end{lem}

The proof is somewhat involved.  The idea of the proof is to use
Lemma~\ref{L_u} to obtain geodesic segments $\gamma_\ell=[u_\ell,v_\ell]$
and $\gamma_r = [u_r,v_r]$ between typical points in the left and right
sides of $\g$, whose intersection $\gamma_\ell\cap\gamma_r$ contains a
large segment from $\gamma$.  Since $\g_\ell$ and $\g_r$ are the unique
geodesics between their (typical) endpoints, we deduce that $\g_n$ contains
$\g_\ell\cap\g_r$ for all large $n$.  See Figure~\ref{F_parallel}.
   
\begin{proof} 
  Let $\g_n=[x_n,y_n]$ and $\g'=[x',y']$, such that 
  $\g_n\to\g'$, as in the lemma be given. 
  
  Let $\eps>0$ and put $\g'_\eps= \g'-(B(x',\eps)\cup B(y',\eps))$. 
  We show that $\g_n$ contains 
  $\g'_\eps$ for all large $n$. Since $\g_n\to\g'$
  (and hence $x_n\to x'$ and $y_n\to y'$)
  this implies that $\g_n\tto\g'$, as required. 
  
  We may assume that $\eps<2^{-1}d(x',y')$.
  Let $u$ (resp.\ $v$)
  denote the point in $\g'$ at distance $\eps/2$ from $x'$ (resp.\
  $y'$).  By Lemma~\ref{L_u}, there are points $u_\ell\in
  B(u,\eps/4)\cap L\cap T$ and $v_\ell\in B(v,\eps/4)\cap L\cap
  T$ such that $[u_\ell,v_\ell]-\g$ is contained in
  $(B(u,\eps/4)\cup B(v,\eps/4))\cap L$.  We also let $u_r,v_r$ be
  defined similarly, replacing $L$ by $R$ everywhere. 
  Note that the
  geodesic segments $[u_\ell,v_\ell]$ and $[u_r,v_r]$ are unique since
  the extremities are all in $T$. Moreover, by our choice of
  $\eps,u,v$, the segments $[u_\ell,v_\ell]$ and $[u_r,v_r]$ intersect
  $\g$ and are disjoint from $\{x',y'\}$.  
  Put
  \[
  \delta=\frac{1}{2}\min\{d(u_\ell,\g),d(v_\ell,\g) ,d(u_r,\g) ,d(v_r,\g)\}
  \]
  and note that $\delta>0$. Let
  $[\g]_\delta=\{z\in M:d(z,\g)< \delta\}$ be the
  $\delta$-neighbourhood of $\g$ in $M$. 

  For $\eta>0$, let us write $V_\eta=\{w\in \CC:d_{\rm E}(w,I)<
  \eta\}$ for the $\eta$-neighbourhood of $I$ in $\CC$. Let $\eta_1>0$ be
  such that
  $V_{\eta_1}\subset \psi([\g]_\delta)$. Such an $\eta_1$ exists
  since, otherwise, we could find a sequence $(z_n)$ of points in $M$
  such that $d(z_n,\g)\geq \delta$ but $d_{\rm
    E}(\psi(z_n),I)\to 0$ as $n\to\infty$, a clear
  contradiction since $\psi(\g)=I$ and $(z_n)$ has convergent
  subsequences.

  Note that
  $\psi(u_\ell),\psi(v_\ell),\psi(u_r),\psi(v_r)\notin V_{\eta_1}$ by the
  definition of $\delta$. Put $I_\ell=\psi([u_\ell,v_\ell])$, and fix
  $\eta_2>0$ such that
  \[
    \eta_2<d_{\rm E}(\psi(x'),I_\ell)\wedge d_{\rm E}(\psi(y'),I_\ell)\,
    ,
  \]
  which is possible since $[u_\ell,v_\ell]$ does not intersect
  $\{x',y'\}$. Finally, we let $\eta_\ell=\eta_1\wedge \eta_2$, and
  similarly define $\eta_r$, and set $\eta = \eta_\ell \wedge \eta_r$.

  Consider $I_\ell$ as a parametrized simple path from $\psi(u_\ell)$ to
  $\psi(v_\ell)$.  This path contains a single segment of $I$, since
  the geodesic 
  $[u_\ell,v_\ell]$ is unique.  Let $u''_\ell,v''_\ell$ be
  defined by $I_\ell \cap I = [\psi(u''_\ell),\psi(v''_\ell)]$,
  with $u_\ell''$ the endpoint closer to $x$.  
  Let the
  last point at which $I_\ell$ enters (the closure of) $V_\eta$ before
  hitting $I$ be $\psi(u'_\ell)$.  Let the first point it exits $V_\eta$
  after separating from $I$ be $\psi(v'_\ell)$.  See
  Figure~\ref{F_parallel}.  Let $H_\ell$ denote the connected component of
  $V_\eta-\psi([u'_\ell,v'_\ell])$ that is contained in $\HH_+$.  Replacing
  $u_\ell,v_\ell$ with $u_r,v_r$ in the arguments above, we obtain $u_r''$,
  $v_r''$, $H_r$.
  Note that our choice of $\eta$ implies that $\psi(x')$ and $\psi(y')$ are
  farther than $\eta$ away (with respect to $d_{\rm E}$) from $H_\ell,H_r$.

  \begin{figure}
    \centering
    \includegraphics[scale=0.8]{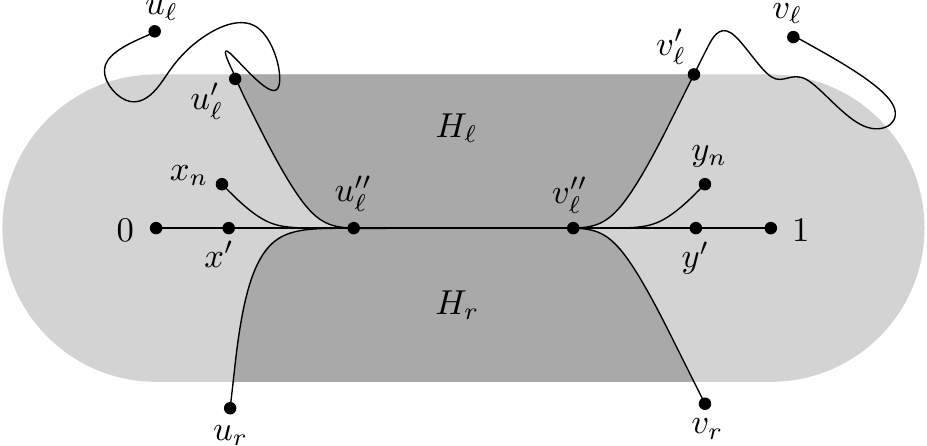}
    \caption{Given $[x',y']\subset \gamma$ we find a geodesic
      $\gamma_\ell =[u_\ell,v_\ell]$ which intersects $\gamma$ in
      $[u''_\ell,v''_\ell]$, which is almost all of $[x',y']$, and
      similarly $[u_r,v_r]$.  These are used to define the sets $V_\eta$
      (shaded), and subsets $H_\ell$ and $H_r$ (dark gray).  For large $n$,
      the geodesics $\gamma_n$ are included in $V_\eta$ and cannot enter
      $H_\ell\cup H_r$, leading to strong convergence.
      The points $u,v,u'_r,u''_r,v'_r,v''_r,u'',v''$ are not shown.  For clarity,
      we omitted $\psi(\cdot)$ from all points (besides
      $\psi(x)=0$ and $\psi(y)=1$) named in the figure.
      }
    \label{F_parallel}
  \end{figure}

  Since $\g_n\to\g'$, we have that for every $n$ large enough,
  $\psi(\g_n)\subset V_\eta$, $\psi(x_n)\in B_{\rm E}(\psi(x'),\eta)$, and
  $\psi(y_n)\in B_{\rm E}(\psi(y'),\eta)$.  By our choice of $\eta$, for
  such an $n$, the extremities $\psi(x_n),\psi(y_n)$ of $\psi(\g_n)$ do not
  belong to $H_\ell\cup H_r$.

  We claim that, for all such $n$, 
  $\psi(\g_n) \cap H_\ell = \emptyset$.  Indeed, 
  if $\psi(\g_n)$ were to intersect
  $H_\ell$, then 
  by the Jordan Curve Theorem 
  it would intersect $\psi([u_\ell',v_\ell'])$ at two points
  $\psi(u_0),\psi(v_0)$ such that the segment 
  $\psi((u_0,v_0))\subset\psi(\g_n)$ is contained in $H_\ell$.
  Since $H_\ell\cap \psi([u_\ell',v_\ell'])=\emptyset$, 
  it would then follow that there are distinct geodesics between 
  $u_0,v_0\in [u_\ell,u_r]$, contradicting the uniqueness $[u_\ell,u_r]$. 
  Similarly, for all such $n$, 
  $\psi(\g_n) \cap H_r = \emptyset$. 
    
  Let $[u'',v'']=[u_\ell'',v_\ell'']\cap[u_r'',v_r'']$, with $u''$ the
  endpoint closer to $x$.  
  Recalling (from the third paragraph of
  the proof) that $d(x',u)=\eps/2$, $d(y',v)=\eps/2$,
  $u_\ell\in B(u,\eps/4)$, $v_\ell\in B(v,\eps/4)$, and
  $[u_\ell,v_\ell]-\g=[u_\ell,u_\ell'') \cup(v_\ell'',v_\ell]$ is contained in
  $B(u,\eps/4)\cup B(v,\eps/4)$, it follows
  that $d(u_\ell'',x'),d(v_\ell'',y')<\eps$. Similarly, 
  since $u_r\in B(u,\eps/4)$, $v_r\in B(v,\eps/4)$, and
  $[u_r,v_r]-\g=[u_r,u_r'') \cup(v_r'',v_r]$ is contained in
  $B(u,\eps/4)\cup B(v,\eps/4)$, 
  we have that 
  $d(u_r'',x'),d(v_r'',y')<\eps$.
  Hence   
  $d(u'',x'),d(v'',y')<\eps$, and so
  $\g'_\eps\subset [u'',v'']$.
  
  To conclude recall that, for all large $n$, we have that 
  $\psi(\g_n)\subset V_\eta$, $\psi(x_n)\in B_{\rm E}(\psi(x'),\eta)$,
  $\psi(y_n)\in B_{\rm E}(\psi(y'),\eta)$, and 
  $\psi(\g_n)\cap (H_\ell\cup H_r)=\emptyset$. 
  By the Jordan Curve Theorem, it moreover follows that
  $[u'',v''] \subset \g_n$, and hence $\g'_\eps \subset \g_n$,
  completing the proof. 
\end{proof}

\begin{proof}[Proof of Proposition~\ref{P_key}]
Since $\g=[x,y]$ is a general geodesic segment from the root of 
the map, we obtain Proposition~\ref{P_key} immediately by 
Lemma~\ref{L_parallel} and invariance under re-rooting.
\end{proof}

With Proposition~\ref{P_key} at hand, Lemma~\ref{L_near} 
follows easily.

\begin{proof}[Proof of Lemma~\ref{L_near}]
  By invariance under re-rooting,  we may restrict to the case that $x$ is the root
  of $\m$.  Let $y\in\m$ and neighbourhoods $N_x$ of $x$ and $N_y$ of $y$
  be given.  Almost surely, there are at most $3$ geodesics from $x$ to
  $y$, which we call $\gamma_i$, for $i=1,\dots, k$ with $k\leq 3$.
  Suppose that $[x_n,y_n]$ is a sequence of geodesic segments with
  $x_n\to x$ and $y_n\to y$ in $(\m,\d)$.  If $[x_{n_k},y_{n_k}]$ is a
  convergent subsequence of $[x_n,y_n]$, then by Lemma~\ref{L_AA},
  $[x_{n_k},y_{n_k}]$ converges to some $\gamma_i$.  By
  Proposition~\ref{P_key}, it follows that 
  $[x_{n_k},y_{n_k}]-(N_x\cup N_y)$ is contained in $\gamma_i$
  for all large $k$. 
  We
  conclude that for any sequence $[x_n,y_n]$ as above,
  for all sufficiently large $n$ we have that $[x_n,y_n]-(N_x\cup N_y)$ is
  contained in some geodesic segment from $x$ to $y$.
  Hence sub-neighbourhoods $N_x'$ and $N_y'$
  as in the lemma exist. 
\end{proof}

\section{Proof of main results}\label{S_proofs}

In this section, we use Proposition~\ref{P_conf-point} to establish our
main results.

\subsection{Typical points}

To simplify the proofs below, we make use of a set
of {\it typical points} $T\subset\m$ (we slightly abuse notation
by keeping the same notation as in Section~\ref{S_near}).  The set $T$
will satisfy the following.
\begin{enumerate}[nolistsep,label={(\roman*)}]
  \item $\lam(T^c)=0$;
  \item Proposition~\ref{P_key} (and weaker
  	results such as Proposition~\ref{P_conf-point}
  	and Lemmas~\ref{L_LG},\ref{L_at},\ref{L_near})
  	holds for all $x\in T$;
  \item Proposition~\ref{P_reg} holds for all $x\in T$;
  \item Proposition~\ref{P_S123} holds for all $x\in T$;
  \item Proposition~\ref{P_ScapU} holds for all $x\in T$;
  \item For each $x,y\in T$, there is a unique geodesic
  	from $x$ to $y$.
\end{enumerate}
To be precise, when we say above that a proposition holds for all $x\in T$, 
we mean that the property in the proposition, known to hold for
$\lambda$-almost every point, holds for every point of $T$.

The almost sure existence of a set $T$ satisfying (i)--(v) 
follows by invariance under re-rooting (and results
cited or proved thus far).  
We note that property (vi) follows by (iii), 
since as mentioned in Section~\ref{S_networks}, if
$(x,y)$ and $(y,x)$ are regular then there is a unique 
geodesic from $x$ to $y$. 

Hence, in the sections which follow, 
to show that various properties hold 
almost surely for $\lambda$-almost every $x\in\m$,
it suffices to confirm that they hold for points in $T$.

\subsection{Geodesic nets}\label{S_net-pf}

Theorems~\ref{T_net-nwd},\ref{T_net-cts}
follow by Proposition~\ref{P_conf-point}.

\begin{proof}[Proof of Theorem~\ref{T_net-nwd}]
  Let $x,y\in\m$ and $u\in T-\{x,y\}$ be given.
  Proposition~\ref{P_conf-point} provides an (open) neighbourhood $U_u$ of
  $u$ and a point $u_0$ outside $U_u$ so that all geodesics from any
  $v\in U_u$ to either $x$ or $y$ pass through $u_0$.  In particular any
  geodesic $[v,x]$, with $v\in U_u$,  can be written as $[v,u_0] \cup [u_0,x]$.  By the
  choice of $u_0$, replacing the
  second segment by some $[u_0,y]$ gives a geodesic from $v$ to $y$.  The same
  holds with $x,y$ reversed.  Consequently, $G(x)\cap U_u = G(y) \cap U_u$.

  Thus $G(x)$ and $G(y)$ coincide in $\bigcup_{T-\{x,y\}} U_u$.
  Since $T$ is dense and has full measure, the theorem follows.
  \end{proof}

\begin{proof}[Proof of Theorem~\ref{T_net-cts}]
  Let $x\in T$ and a neighbourhood $N$ of $x$ be given. 
  Select $\eps>0$ so that $B(x,2\eps)\subset N$. 
  Let $N'\subset B(x,\eps)$ and $x_0\in B(x,\eps)-N'$
  be as in Proposition~\ref{P_conf-point}. By the choice of $x_0$, 
  for any $y_0\in N^c$ 
  and $x'\in N'$, observe that $y_0\in \geo(x')$ if and only if there 
  is some $y\in B(x,\eps)^c$ and geodesic $[x_0,y]$ so that 
  $y_0\in[x_0,y)$. 
  This condition is independent of $x'$.   
  Hence all $\geo(x')$, $x'\in N'$, coincide on $N^c$.
\end{proof}
  
In support of our conjecture in 
Section~\ref{S_conf-point}, we show
that the union of most geodesic nets is of Hausdorff dimension 1.

\begin{pro}\label{P_netLam}
  Almost surely, there is a subset $\Lambda\subset\m$
  of full volume, $\lam(\Lambda^c)=0$, satisfying 
  $\H\bigcup_{x\in \Lambda}\geo(x)=1$.
\end{pro}

\begin{proof}
  We prove the claim with $\Lambda=T$, which has full measure.
  
  By property (ii) of points in $T$, 
  there is a confluence of geodesics to all points $x\in T$
  (that is, the statement of Lemma~\ref{L_LG} holds).
  As discussed in Section~\ref{S_net}, 
  we thus have that
  $\H G(x)=1$ for all $x\in T$.
  
  Let $\eps>0$ be given. For each $x\in T$, put
  $\geo_\eps(x)=\geo(x)-B(x,\eps)$.  By Theorem~\ref{T_net-cts}, for each
  $x\in T$ there is an $\eta_{x}\in(0,\eps)$ such that
  $\geo_{2\eps}(x')\subset \geo_{\eps}(x)$ for all $x'\in
  B(x,\eta_x)$. Since $(\m,d)$ is a separable metric space and hence
  strongly Lindel{\"o}f (that is, all open subspaces of $(\m,\d)$ are
  Lindel{\"o}f) there is a countable subset $T_\eps\subset T$ such that
  $\bigcup_{x\in T_\eps} B(x,\eta_x)$ is equal to
  $\bigcup_{x\in T} B(x,\eta_x)$, and in particular, contains $T$. Hence,
  by the choice of $T_\eps$, 
  $\bigcup_{x\in T}\geo_{2\eps}(x)$ is contained in 
  $\bigcup_{x\in T_\eps}\geo_{\eps}(x)$, 
  a countable union of 1-dimensional sets, and so is 1-dimensional.
  
  Taking a countable union over $\eps=1/n$, we see that
  $\H\bigcup_{x\in T}\geo(x)=1$, which yields the claim.
\end{proof}

\subsection{Cut loci}\label{S_cut-pf} 

As discussed in Section~\ref{S_cut}, Le Gall's study of geodesics reveals a
correspondence between cut-points of the CRT and points with multiple
geodesics to the root of the Brownian map.  Hence, Le Gall~\cite{LG10}
states that $S(\rho)$ ``exactly corresponds to the cut locus of [the
Brownian map] relative to the root.''

\subsubsection{Weak cut loci}

The main way in which the weak cut locus is badly behaved is that there is
a dense set of points for which the weak cut locus has positive volume and
full dimension (whereas typically it is much smaller, 
see Proposition~\ref{P_Dim-S}).

\begin{pro}\label{P_mul}
  Almost surely, for $\lam $-almost every $x\in\m$, for any neighbourhood
  $N$ of $x$, there is a set $D$ with $\H D=2$, dense in some neighbourhood
  $N'\subset N$ of $x$, such that $N^c \subset S(x')$ for all $x'\in D$.
\end{pro}

\begin{proof}
  Let $x\in T$ and a neighbourhood $N$ of $x$ be given. Let $N'\subset N$
  and $x_0\in N-N'$ be as in Proposition~\ref{P_conf-point}.  Fix some
  $u\in N^c\cap T$, and put $D=N'\cap S(u)$ so that by properties (iv),(v)
  of points in $T$, we have that $D$ is dense in $N'$ and satisfies
  $\H D=2$.  By property (vi) of points in $T$, there is a unique geodesic
  from $u$ to $x$. Since this geodesic passes through $x_0$, it follows
  that there is a unique geodesic from $u$ to $x_0$. Hence, by the choice
  of $D$ and $x_0$, we see that there are multiple geodesics from each
  point $x'\in D$ to $x_0$. We conclude, by the choice of $x_0$,
   that $N^c\subset S(x')$, for all
  $x'\in D$.
\end{proof}

Since the weak cut locus relation is symmetric --- that is,
$y\in S(x)$ if and only if $x\in\mul(y)$ --- we note 
that it follows immediately by Proposition~\ref{P_mul} that 
almost surely, for \emph{all} $x\in\m$, $S(x)$ is dense in $\m$ 
(as mentioned in Section~\ref{S_cut}) and $\H S(x)\ge 2$.

By the proof of Proposition~\ref{P_mul}, we find that 
$S(x)$ does not effectively capture the essence of a 
cut locus of a general point $x\in\m$. Therein, observe that 
although all points $y\in N^c$ are in $S(x')$, $x'\in D$, 
this is due to the structure of the map near $x'$ (namely the multiple
geodesics to the confluence point $x_0$) 
and does not reflect on the map near $y$.  
For this reason, we also define a strong cut locus for the 
Brownian map, see Section~\ref{S_cut}.

\subsubsection{Strong cut loci}

By Le Gall's description of geodesics to the root and
invariance under re-rooting, and in particular
Proposition~\ref{P_reg}, 
we immediately obtain the following:

\begin{pro}\label{P_S=C}
   Almost surely, for $\lam$-almost every $x\in\m$, 
   $\mul(x)=\cut(x)$, that is, the weak and strong cut loci coincide.
\end{pro}

We remark that the strong cut locus relation, unlike the 
weak cut locus, is not symmetric in $x$ and $y$, that is, 
$y\in\cut(x)$ does not imply that $x\in\cut(y)$. See Figure~\ref{F_cut}.

\begin{figure}
  \centering
  \includegraphics[scale=0.8]{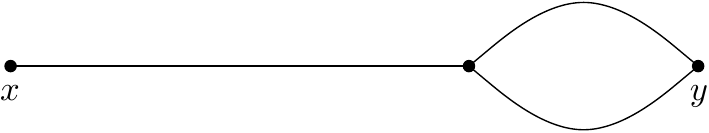}
  \caption{ Asymmetry of the strong cut locus relation: 
  For a regular pair $(x,y)$ joined by two geodesics, we have
  $y\in\cut(x)$, however $x\notin\cut(y)$, since all
    geodesics from $y$ to $x$ coincide near $x$.}
  \label{F_cut}
\end{figure}

Although more in tune with the singular geometry of 
the Brownian map, not all properties of cut loci in smooth manifolds
apply for the Brownian map.
For instance, $\cut(x)$ is much smaller than the \emph{closure} 
of all points with multiple geodesics to $x$ 
(as is the case with the cut locus 
of a smooth surface, see Klingenberg~\cite[Section 2.1.14]{K95})
since the set of such points is dense in $\m$ 
(as noted after the proof of Proposition~\ref{P_mul}). 
Moreover, it is not necessarily the case that all points $y\in\cut(x)$ 
are endpoints relative to $x$ (that is, extremities $y$ 
of a geodesic $[x,y]$ which cannot be extended
to a geodesic $[x,y']\supset[x,y]$ for any $y'\neq y$; 
in other words, $y\notin G(x)$). 
For instance,
if $\gamma,\gamma'$ are distinct 
geodesics from the root of the map $\rho$
to some
point $x$, 
with a common initial segment $[\rho,y]=\gamma\cap\gamma'$,
then note that $y$ is in $\cut(x)$ (by Proposition~\ref{P_reg}), 
however not an endpoint relative to 
$x$, being  in the relative interior of $\g$. 

Despite such differences, we propose that the set $\cut(x)$ 
is a more interesting notion of cut locus in our setting than 
$S(x)$ or, say, the set of \emph{all} endpoints relative to $x$
(that is, $\geo(x)^c-\{x\}$), which by Theorem~\ref{T_frame} is a 
residual subset of the map.

As stated in Section~\ref{S_cut}, analogues of
Theorems~\ref{T_net-nwd},\ref{T_net-cts} hold for the strong cut locus.
The proofs are very similar to those 
of Theorems~\ref{T_net-nwd},\ref{T_net-cts}.

\begin{proof}[Proof of Theorem~\ref{T_cut-nwd}]
  Let $x,y\in\m$ and $u\in T-\{x,y\}$ be given.
  Proposition~\ref{P_conf-point} provides an (open) neighbourhood $U_u$ of
  $u$ and a point $u_0$ outside $U_u$ so that all geodesics from any
  $v\in U_u$ to either $x$ or $y$ pass through $u_0$.  In particular any
  geodesic $[v,u_0]$ can be extended to each of $x,y$.  

  Since $v\in\cut(x)$ is determined by the structure of geodesics $[v,x]$
  near $v$, a point $v\in U_u$ is in $\cut(x)$ if and only if
  $v\in\cut(y)$.  Thus $\cut(x)$ and $\cut(y)$ agree in
  $\bigcup_{u\in T-\{x,y\}} U_u$.  The result follows, since $T$ is
  dense and has full measure.
\end{proof}

\begin{proof}[Proof of Theorem~\ref{T_cut-cts}] 
  Let $x\in T$ and a neighbourhood $N$ of $x$ 
  be given. Let $N'\subset N$ and $x_0\in N-N'$ be as in 
  Proposition~\ref{P_conf-point}. For any $x'\in N'$ and 
  $y\in N^c$, $y\in\cut(x')$ if and only if there are multiple 
  geodesics from $x_0$ to $y$ which are distinct near $y$.
  Since this condition is independent of $x'$, 
  we conclude that all $\cut(x')$, $x'\in N'$, 
  coincide on $N^c$.
\end{proof}

Analogously to Proposition~\ref{P_netLam}, 
we find that  
the union over most strong cut loci is of Hausdorff 
dimension 2.

\begin{pro}\label{P_cutLam}
  Almost surely, there is a subset $\Lambda\subset\m$
  of full volume, $\lam(\Lambda^c)=0$, satisfying 
  $\H\bigcup_{x\in \Lambda}\cut(x)=2$.
\end{pro}

\begin{proof}
  The proposition follows by the proof of 
  Proposition~\ref{P_netLam}, but replacing its use of 
  Theorem~\ref{T_net-cts} with that of Theorem~\ref{T_cut-cts},
  and noting, by property (iv) of points in $T$, that
  $\H C(x)=2$
  for all $x\in T$.
  We omit the details.
\end{proof}

It would be interesting to know if almost surely 
$\bigcup_{x\in\m}\cut(x)$ is of Hausdorff dimension 2.

\subsection{Geodesic stars}\label{S_stars} 

A \emph{geodesic star} is a formation of geodesic
segments which share a common endpoint and are 
otherwise pairwise disjoint. Geodesic stars play a 
important role in~\cite{M13}.  While every point is the centre of a
geodesic star with a single ray, almost every point is not the centre of a
star with any more rays.

\begin{defn}
  For $\eps>0$, let $\star(\eps)$ denote the 
  set of points $x\in\m$ such that for some 
  $y,y'\in B(x,\eps)^c$ and geodesic segments $[x,y]$ and 
  $[x,y']$, we have that $(x,y]\cap(x,y']=\emptyset$.
  We call a point in $\star(\eps)$ the \emph{centre of a geodesic $\eps$-star
  with two rays}.
\end{defn}

Note that any point in the interior of a geodesic is in
$\star(\eps)$ for some $\eps>0$, but the converse need not hold.

\begin{pro}\label{P_stars}
  Almost surely, for any $\eps>0$, 
  $\star(\eps)$ is nowhere dense in $\m$.
\end{pro}

\begin{proof}
  Let $\eps>0$ and $x\in T$ be given. 
  Put $N=B(x,\eps/2)$. Let $N'\subset N$ and $x_0\in N-N'$ 
  be as in Proposition~\ref{P_conf-point}. Since 
  $N\subset B(x',\eps)$ for all $x'\in N'$, 
  $x_0$ 
  is contained in all geodesic segments of length $\eps$ 
  from points $x'\in N'$. 
  Hence $\star(\eps)\cap N'=\emptyset$. 
  The result thus follows by the density of $T$.
\end{proof}

\begin{proof}[Proof of Theorems~\ref{T_frame},\ref{T_locus}]
  Note that if a point is either in the relative interior of 
  a geodesic or in the strong cut locus
  of a point, then it is the centre of a 
  geodesic $\eps$-star with two rays, for some $\eps>0$.
  Therefore 
  $\bigcup_{x\in\m}\geo(x)$ and
  $\bigcup_{x\in\m}\cut(x)$
  are contained in
  $\bigcup_{n\ge1}\star(n^{-1})$, a set of first Baire 
  category by Proposition~\ref{P_stars}.
  The theorems follow.
\end{proof}

\subsection{Geodesic networks}\label{S_networks-pf} 

In this section, we classify the types of geodesic 
networks which are dense in the Brownian map
and calculate the dimension of the set of pairs
with each type of network.

\begin{proof}[Proof of Theorem~\ref{T_normal}]
  Let $u\neq v\in T$ be given. 
  By property (vi) of points in $T$, there is a unique
  geodesic $[u,v]$.  Put $\eps=\frac13 d(u,v)$.
  By property (ii) of points in $T$, we have
  by Lemma~\ref{L_near} that there is an $\eta>0$ so 
  that if $U=B(u,\eta)$ and $V=B(v,\eta)$, then for any 
  $u'\in U$ and $v'\in V$, any geodesic segment $[u',v']$ 
  coincides with $[u,v]$ outside of 
  $B(u,\eps)\cup B(v,\eps)$.
  
  Let $z$ denote the midpoint of $[u,v]$.
  By the choice of $\eta$ and since $u\in T$, 
  we have by properties (iii),(iv) for points in $T$
  that for all $v'\in V$, the pair $(z,v')$ is 
  regular and joined by at most three geodesics.
  Hence we split $V = V_1\cup V_2\cup V_3$, 
  where $V_k$ consists of $v'\in V$ for which 
  $(z,v')\in N(1,k)$. Similarly, we decompose 
  $U=U_1\cup U_2 \cup U_3$ according to the 
  number of geodesics between $z$ and $u'\in U$. 
  Since $u,v\in T$, 
  we see by property (iv) of points in $T$ that
  all $U_j,V_k$ are dense in $U,V$. 
  
  Finally, by the choice of 
  $\eta$, observe that $U_j\times V_k\subset\net(j,k)$, 
  for all $j,k\in\{1,2,3\}$. Hence, parts (i),(ii) 
  of the theorem follow by the density of $T$.
\end{proof}

For the proof of Theorem~\ref{T_normal-dim},
we require the following result concerning the dimension 
of cartesian products in arbitrary metric spaces.

\begin{lem}[{Howroyd~\cite{H95,H96}}]\label{L_dimx}
  For any metric spaces $X,Y$ we have that
  \begin{enumerate}[nolistsep,label={(\roman*)}]
  \item $(\H X) + (\H Y) \le \H (X\times Y)$;
  \item $\pack (X\times Y) \le (\pack X) + (\pack Y)$,
  \end{enumerate}
  where the metric on $X\times Y$ is the $L^1$ metric on the product.
\end{lem}

\begin{proof}[Proof of Theorem~\ref{T_normal-dim}]
  Let $u\neq v\in T$ and $U_j,V_k$, $j,k\in\{1,2,3\}$, be as in the proof
  of Theorem~\ref{T_normal}.  Since $u,v\in T$, we have by properties
  (iv),(v) of points in $T$ that for all $j,k\in\{1,2,3\}$,
  $\H U_j=\pack U_j=2(3-j)$, $\H V_k=\pack V_k=2(3-k)$, and moreover, 
  the sets $U_3,V_3$ are countable.
  
  Recall that in the proof of Theorem~\ref{T_normal}, it is shown that for
  all $j,k\in\{1,2,3\}$, $U_j\times V_k\subset N(j,k)$.  We thus obtain the
  lower bounds $\H N(j,k)\ge 2(6-j-k)$ by 
  Lemma~\ref{L_dimx}(i).  In
  particular, since $\H A \le \pack A$, we obtain
  $8\leq \H N(1,1)\leq \pack N(1,1)\leq \pack M^2\leq 8$, 
  where the  last
  inequality follows by Proposition~\ref{P_Dim-U} and
  Lemma~\ref{L_dimx}(ii).  Hence, we find that
  $\dim N(1,1)=\pack N(1,1)=8$.

  It remains to give an upper bound on the dimensions 
  of $N(j,k)$ when
  $j,k$ are not both $1$, in which case 
  the complement of the geodesic
  network $G(x,y)$ is disconnected.  
  By symmetry, we assume $j\neq 1$, so
  that there are multiple geodesics leaving $x$.  
  Let $[x',y']$ be the
  intersection of all geodesics $[x,y]$.  
  (If $k=1$, then we have that $y'=y$.)

  Fix a countable, dense subset $T_0\subset T$.  Take some $x_0\in T_0$ in
  a component $U_x$ of $G(x,y)^c$ whose closure contains
  $x$ but not $[x',y']$. (See
  Figure~\ref{F_normal-dim}.)  By
  the Jordan Curve Theorem and the choice of $[x',y']$, 
  for any geodesic $[x_0,y]$
  we have that $[x_0,y]-U_x$ is contained in some geodesic
  from $x$ to $y$, and in particular, contains $[x',y']$.  
  Since $x_0$ is
  typical, by property (ii) of points in $T$, 
  we have that all sub-segments of all geodesics
  $[x_0,y]$ are stable.
  Let $z$ denote the midpoint of $[x',y']$.
   Note that, in particular, $[x',z]\subset[x',y']$ and
  $[z,y']\subset[x',y']$ are stable. 

  \begin{figure}
    \centering
    \includegraphics[scale=0.8]{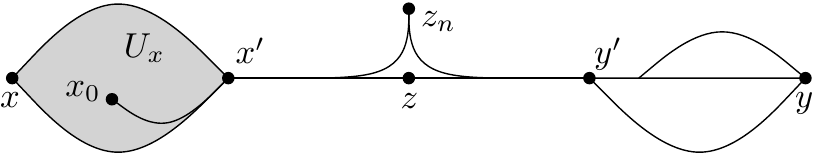}
    \caption{Theorem~\ref{T_normal-dim}:
    As depicted, $(x,y)\in N(2,3)$.  A typical point
      $x_0\in U_x$ 
      gives normal geodesics $[x_0,y]$. For 
      some $z_n\in T_0$ sufficiently close
      to $z$, we have that $(z_n,x)\in N(1,2)$ and
      $(z_n,y)\in N(1,3)$, and hence
      $(x,y)\in S_2(z_n)\times S_3(z_n)$.}
    \label{F_normal-dim}
  \end{figure}

  Take a sequence of points $z_n\in T_0$ converging to $z$.  
  Any subsequential limit of geodesics $[x,z_n]$ converges to
  some geodesic $[x,z]$, which, by the choice of $[x',y']$,
   contains $[x',z]$. Since
  $[x',z]$ is stable,
  for large enough $n$ the geodesics
  $[x,z_n]$ intersect $[x',z]$, and therefore (viewing $[x,z_n]$
  as parametrized from $x$ to $z_n$) necessarily  
  coincide with one of the
  geodesics $[x,x']$,
   and then continue along $[x',y']$ before branching off
  towards $z_n$.  It follows that 
  for such $n$, we have that $(x,z_n) \in N(j,1)$.
  Similarly, since $[z,y']$ is stable, for large enough $n$ the geodesics
  $[z_n,y]$ all go through $y'$, and hence $(z_n,y) \in N(1,k)$.

  By property (iii) of points in $T$, we note that for 
  any $u\in T$ and $i\in\{1,2,3\}$, $S_i(u)$ (as defined in 
  Section~\ref{S_dims}) is equal to $\{v : (u,v) \in N(1,i)\}$.
  Furthermore, by properties (iv),(v) of points in $T$, 
  we have that 
  $\pack S_i(u) = 6-2i$, and moreover, $S_3(u)$ is countable.
  
  The above argument
  shows that for every $(x,y) \in N(j,k)$ we have 
  that $(z_n,x)\in N(1,j)$ and
  $(z_n,y)\in N(1,k)$ for some $z_n\in T_0$.  Thus
  \[
    N(j,k) \subset \bigcup_{u\in T_0} S_i(u)\times S_j(u).
  \]
  Therefore, since $T_0$ is countable, we see by 
  Lemma~\ref{L_dimx}(ii) that
  $\pack N(j,k) \leq (6-2j)+(6-2k)$, 
  giving the requisite upper bound. 
  Moreover, we find that 
  $N(3,3)$ is countable. 
  
  Altogether, since $\H A\le \pack A$, we conclude that 
  $N(j,k)$ has Hausdorff and packing dimension
$2(6-j-k)$.
\end{proof}

\begin{proof}[Proof of Corollaries~\ref{T_pairs},\ref{T_pairs-dim}]
  Noting that $ N(j,k)\subset\pair(jk)$, 
  for all $j,k\in\NN$, we observe that 
  Theorems~\ref{T_normal},\ref{T_normal-dim} 
  immediately yield 
  Corollaries~\ref{T_pairs},\ref{T_pairs-dim}.
\end{proof}

\section{Related models}\label{S_models}

Our results have implications for the geodesic structure
of models related to the Brownian map.

An infinite volume version of the Brownian map, 
the \emph{Brownian plane} $(P,D)$, has been introduced 
by Curien and Le Gall~\cite{CLG13}. The random metric 
space $(P,D)$ is homeomorphic to the plane $\RR^2$ and 
arises as the local Gromov-Hausdorff scaling limit of the 
UIPQ (discussed in Section~\ref{S_net}). 
The Brownian plane has an additional scale invariance 
property which makes it more amenable to analysis, see 
the recent works of Curien and Le Gall~\cite{CLG14a,CLG14b}. 
As discussed in~\cite{LG14}, almost surely there are 
isometric neighbourhoods of the roots of $(\m,\d)$ and 
$(P,D)$. Using this fact and scale invariance, properties 
of the Brownian plane can be deduced from 
those of the Brownian map.

In a series of works, Bettinelli~\cite{B10,B12,B16} 
investigates \emph{Brownian surfaces} of positive genus. 
In~\cite{B10} subsequential Gromov-Hausdorff 
convergence of uniform random bipartite quadrangulations 
of the $g$-torus ${\mathbb T}_g$ is established (also 
general orientable surfaces with a boundary are analyzed 
in~\cite{B16}), and it is an ongoing work of 
Bettinelli and Miermont~\cite{BM15a,BM15b} to confirm
that a unique scaling limit exists. Some properties hold 
independently of which 
subsequence is extracted. For instance, a scaling limit of 
bipartite quadrangulations of ${\mathbb T}_g$ is homeomorphic 
to ${\mathbb T}_g$ (see~\cite{B12}) and has Hausdorff dimension 
4 (see~\cite{B10}). Also, a confluence of geodesics is observed 
at typical points of the surface (see~\cite{B16}). Our results imply 
further properties of geodesics in such surfaces, 
although in these settings there are additional technicalities 
to be addressed.

\section*{Acknowledgements}

OA and BK thank GM and UMPA for support and hospitality during a visit to
ENS Lyon, when this project was initiated.  We thank the Isaac Newton
Institute for Mathematical Sciences, Cambridge, where this project was
completed during the program ``Random Geometry'', supported by EPSRC Grant
Number EP/K032208/1.  We also thank the referee for a careful reading and
suggestions for improvements.  OA was supported by NSERC of Canada and the Simons
Foundation.  BK was supported by NSERC of Canada, Killam Trusts and a Michael Smith
Foreign Study Supplement.  GM was supported by 
 the Grant ANR-14-CE25-0014 (ANR GRAAL).

\providecommand{\bysame}{\leavevmode\hbox to3em{\hrulefill}\thinspace}
\providecommand{\MR}{\relax\ifhmode\unskip\space\fi MR }
\providecommand{\MRhref}[2]{%
  \href{http://www.ams.org/mathscinet-getitem?mr=#1}{#2}
}
\providecommand{\href}[2]{#2}

\end{document}